\documentclass[11pt,a4paper]{amsart}

\usepackage[utf8]{inputenc}
\usepackage[T1]{fontenc}
\usepackage[english]{babel}

\usepackage{graphicx, pstricks, pst-plot, pst-math,esint} 
\usepackage{amsfonts}
\usepackage{amsmath, latexsym}
\usepackage{amssymb,verbatim}
\usepackage{hyperref,pdfsync}
\usepackage{graphicx}
\usepackage{color}
\usepackage{mathtools}

\newtheorem{lem}{Lemma}
\newtheorem{pro}{Proposition}
\newtheorem{thm}{Theorem}
\newtheorem{cor}{Corollary}
\newtheorem{df}{Definition}

\newtheorem{rem}{Remark}

\newcommand{\FF}{{\mathcal F}}

\def\XXint#1#2#3{{\setbox0=\hbox{$#1{#2#3}{\int}$}
     \vcenter{\hbox{$#2#3$}}\kern-.5\wd0}}

\newcommand {\den} {{\bf e}_\eps}
\newcommand {\ka}{\kappa}
\newcommand {\bun}{\beta_1}
\newcommand {\bd}{\beta_2}
\newcommand {\cd}{\mathcal{D}}

\newcommand{\RR}{\mathbb{R} }
\newcommand{\R}{\mathbb{R}}
\newcommand{\Z}{\mathbb{Z}}

\newcommand{\eps}{\varepsilon}

\newcommand{\ds}{\displaystyle}
\newcommand{\h}{{\mathcal{H}}}
\newcommand{\p}{{\mathcal{P}}}
\newcommand{\Hh}{{\mathcal{H}}}

\newcommand{\be}{\begin{equation}}
\newcommand{\ee}{\end{equation}}

\newcommand{\nd}{\noindent}

\newcommand{\Om}{\Omega}
\newcommand{\om}{\omega}

\newcommand{\hul}{\mathop{\rm hull\, }}

\newcommand{\degr}{\operatorname{deg}}

\newcommand{\dt}{\partial_t}

\newcommand{\tendf}{\rightharpoonup}

\newcommand{\ma}{{m_{1,\infty}}}

\newcommand{\md}{m_\delta}
\newcommand{\vd}{v_\delta}

\newcommand{\Ss}{\mathbb{S}}
\newcommand{\Tt}{\mathbb{T}}

\newcommand{\loc}{\mathrm{loc}}

\newcommand{\q}[1]{\mathcal{#1}}
\newcommand{\m}[1]{\mathbb{#1}}
\newcommand{\mr}[1]{\mathrm{#1}}
\newcommand{\e}{\varepsilon}

\begin{document}

\title[A thin-film limit for N\'eel walls in LLG equation]{A thin-film limit in the Landau-Lifshitz-Gilbert equation relevant for the formation of N\'eel walls }

\author{Rapha\"el C\^ote}
\address{Rapha\"el C\^ote \and Evelyne Miot \\
Centre de Math\'ematiques Laurent Schwartz, UMR 7640, \'Ecole Polytechnique \\ 
91128 Palaiseau, France}
\email{cote@math.polytechnique.fr \\
evelyne.miot@math.polytechnique.fr}
\author{Radu Ignat}
\address{Radu Ignat \\ Institut de Math\'ematiques de Toulouse, Universit\'e Paul Sabatier \\
118 Route de Narbonne, 31062 Toulouse, France}
\email{radu.ignat@math.univ-toulouse.fr} 
\author{Evelyne Miot}
\dedicatory{Dedicated to our Professor Ha\"im Brezis on his seventieth anniversary\\ with esteem
}


\begin{abstract} We consider an asymptotic regime for two-dimensional ferromagnetic films that is consistent with the formation of  transition layers, called 
N\'eel walls. We first establish compactness of $\Ss^2$-valued magnetizations in the energetic regime of N\'eel walls and characterize the set of accumulation points. We then prove that N\'eel walls are asymptotically the unique energy minimizing configurations. We finally study the corresponding dynamical issues, namely the compactness properties of the magnetizations under the flow of the Landau-Lifshitz-Gilbert equation.
\end{abstract}

\maketitle

\section{Introduction and main results}

The purpose of this paper is to study an asymptotic regime for two-dimensional ferromagnetic thin films allowing for the occurrence and persistence of special transition layers called N\'eel walls. We will prove compactness, optimality and energy concentration of N\'eel walls, together with dynamical properties driven by the Landau-Lifschitz-Gilbert equation.

\subsection{A two-dimensional model for thin-film micromagnetics} We will focus on the following $2D$ model for thin ferromagnetic films. For that, let
$$\Omega=\RR \times \Tt \quad \textrm{with} \quad \Tt=\RR/\mathbb{Z},$$ be a two-dimensional horizontal section of a magnetic sample that is infinite in $x_1$-direction and periodic in $x_2$-direction. The admissible magnetizations are vector fields
\[ m=(m',m_3): \Omega \to \mathbb{S}^2, \quad m'=(m_1, m_2),\]
that are periodic in $x_2$-direction (this condition is imposed in order to rule out lateral surface charges) and connect two macroscopic directions forming an angle, i.e., for a fixed $\ma\in [0,1)$,
\be \label{1b} 
m(x_1, x_2)= \begin{cases}
m_{-\infty} & \text{for } x_1 \le -1, \\
m_{+\infty} & \text{for } x_1 \ge 1,
\end{cases} \quad \text{where }\quad m_{\pm \infty} = \begin{pmatrix}
\ma \\  \pm \sqrt{1-\ma^2}\\0
\end{pmatrix}.
\ee
We will consider the following micromagnetic energy approximation in a thin-film regime that is written in the absence of crystalline anisotropy and external magnetic fields (see e.g. \cite{DKMO_reduced}, \cite{Kohn-Slastikov}):
\be \label{definitia} 
E_\delta(m)=
\int_{\Omega}\bigg( |\nabla m|^2  + \frac{1}{\eps^2} m_3^2\, \bigg)dx+ \frac 1 \delta  \int_{\Omega\times \R} |h(m')|^2 \, dx dz , \ee 
where $\delta>0$ and $\eps=\eps(\delta)>0$ are two small parameters.  The first term in \eqref{definitia} is called the
exchange energy, while the other two terms stand for the stray field energy created by the surface charges $m_3$ at the top and bottom of the sample and by the volume charges $\nabla\cdot m'$ in the interior of the sample. More precisely, the
stray-field $h(m'):\Omega\times \R \to \R^3$ generated only by the volume
charges is defined as
the unique $L^2(\Omega\times \R, \R^3)-$gradient field
\begin{equation*}
  h(m')=(\nabla, \frac{\partial}{\partial z}) U(m')
\end{equation*}
that is $x_2$-periodic and is determined by static Maxwell's equation in the weak sense\footnote{In other words, $h(m')$ is the Helmholtz projection of the vector measure $m' \h^2\llcorner \Omega\times \{0\}$ onto the $L^2(\Omega\times \R)$-space of gradient fields. }: For all $\zeta \in
C_c^\infty(\Omega\times \R)$,
\begin{align}
\label{eqh}
  \int_{\Omega\times \R} (\nabla, \frac{\partial}{\partial z}) U(m') \cdot (\nabla,
  \frac{\partial}{\partial z}) \zeta \ dx dz = \int_{\Omega} m' \cdot \nabla \zeta \
  dx.
\end{align}
Explicitly solving \eqref{eqh} by use\footnote{Given 
a function $\zeta:\Omega\to \R$ which is $1$-periodic in $x_2$, we introduce the combination of Fourier 
transformation in $x_1$ and Fourier series in $x_2$ by 
$\nd \FF(\zeta)(\xi) = \frac{1}{\sqrt{2\pi}} \int_{\Omega} e^{-i\xi\cdot x} \zeta(x)\, dx$ where $\xi\in\R\times {2\pi}\Z$. } of the Fourier transform $\FF(\cdot)$, the stray-field energy can be
equivalently expressed in terms of the homogeneous $\dot H^{-1/2}-$norm of $\nabla
\cdot m'$ (see e.g. \cite{IK})\footnote{One computes that $\FF(U(m')(\cdot, z))(\xi)=-\frac{1}{2|\xi|}e^{-|\xi|\, |z|} \FF(\nabla \cdot m')(\xi)$ for $\xi\neq 0$ and $z\in \R$.}:
\begin{align}
\label{eqh2}
  \int_{\Omega\times \R} |h(m')|^2 \, dx dz=\frac 1 2 \int_{\R\times {2\pi}\Z} \frac{1}{|\xi|}
  |\FF(\nabla \cdot m')(\xi)|^2\, d\xi=\frac 1 2 \int_{\Omega}||\nabla|^{1/2}\Hh(m')|^2\, dx,\end{align}
   where
\[ \Hh(m')=-\nabla (-\Delta)^{-1}\nabla \cdot m', \quad \textrm{i.e., }\quad \FF(\Hh(\cdot))(\xi)=\frac{\xi\otimes\xi}{|\xi|^2}, \quad \xi\in \R\times {2\pi}\Z\setminus\{(0,0)\},\]
so that the gradient of the energy $E_\delta(m)$ is given by
\begin{equation} \label{nablaE}
\nabla E_{\delta}(m)=- 2 \Delta m+\bigg(\frac{1}{\delta}(-\Delta)^{1/2}  \Hh(m'),\frac{2m_{3}}{\eps^2}\bigg).
\end{equation}
Here{ and in the following}, we denote planar coordinates by $x=(x_1, x_2)$,
$(x_1,x_2)^\perp =(-x_2,x_1)$, the vertical coordinate by $z$ and furthermore, we
write $(\nabla, \frac{\partial}{\partial z})=( \frac{\partial}{\partial x_1},
\frac{\partial}{\partial x_2}, \frac{\partial}{\partial z})$ and $\Delta=\frac{\partial^2}{\partial x_1^2}+\frac{\partial^2}{\partial x_2^2}$.

\bigskip
In this model, we expect two types of singular patterns: N\'eel walls and vortices (so-called Bloch lines in micromagnetic jargon). These patterns result from the competition between the different contributions in the total energy $E_\delta(m)$ and the nonconvex constraint $|m|=1$. We explain
 these structures in the following and compare their respective energies (for more details, see DeSimone, Kohn, M\"uller and Otto \cite{DeSKMO}).
 
 \medskip

 {\nd \bf N\'eel walls.} 
The N\'eel wall is a dominant transition layer in thin
ferromagnetic films. It is characterized by a one-dimensional
in-plane rotation connecting two directions \eqref{1b} of the
magnetization. More precisely, it is a one-dimensional transition
$m=(m_1, m_2):\RR \to \Ss^1$ that minimizes the energy under the boundary constraint \eqref{1b}:
$$E_{\delta}(m)=\int_{\RR} \bigg|\frac{dm}{dx_1} \bigg|^2\, dx_1+
\frac{1}{2\delta} \int_{\RR} \bigg|\, \bigg|\frac{d}{dx_1}\bigg|^{1/2}m_1 \bigg|^2 \, dx_1.$$ 
It follows that the minimizer is a two length scale object: it has a small core with fast
varying rotation and two logarithmically decaying tails.\footnote{In our model, the tails are contained by the system thanks to the confining mechanism of steric interaction with the sample
edges placed at $x_1=\pm 1$.} As $\delta\to 0$, the scale of the N\'eel core is given by $|x_1|\lesssim w_{core}=O(\delta)$ (up to a logarithmic scale in $\delta$) while the two logarithmic
 decaying tails scale as $w_{core}\lesssim|x_1|\lesssim w_{tail}=O(1)$. The energetic cost (by unit length)
 of a N\'eel wall is given by
 \begin{align*}
   E_\delta(\text{N\'eel wall})=\frac{\pi(1-\ma)^2+o(1)}{2\delta |\log \delta|} \quad \textrm{as} \quad \delta \to 0,
 \end{align*}
 (see e.g. \cite{DeSKMO}, \cite{I}).

 \medskip

 {\nd \bf Micromagnetic vortex.} A vortex point corresponds in our model to a topological singularity at the microscopic level where the
 magnetization points out-of-plane. The prototype of a vortex configuration is given by a vector field
 $m: B^2\to \Ss^2$ defined in a unit disk $\Omega=B^2$ of
 a thin film that satisfies:
 \begin{align*}  %
   \nabla\cdot m'=0 \textrm{ in } B^2 \quad \text{ and } \quad m'(x)={x^\perp} \textrm{ on
   }\partial B^2
 \end{align*}
 and minimizes the energy \eqref{definitia}:\footnote{In our model, the parameter $\eps=\eps(\delta)>0$ is related to $\delta$ by the regime \eqref{reg_eta_eps}. }
 $$
E_\delta(m)=\int_{B^2} |\nabla m|^2 \, dx +\frac{1}{\eps^2} \int_{B^2} m^2_3\, dx.$$
Since the magnetization turns in-plane at
 the boundary of the disk $B^2$ (so, $\degr(m', \partial \Omega)=1$), a localized region is created, that is the core of the vortex of size $\eps$, where the magnetization becomes indeed perpendicular to the horizontal plane. Remark that the energy $E_\delta$ controls the Ginzburg-Landau energy, i.e., 
\[ E_\delta(m)\geq \int_{B^2} \den(m')\, dx, \, \textrm{ with }\den(m')= |\nabla m'|^2+\frac{1}{\eps^2} (1-|m'|^2)^2 \]
since $|\nabla (m', m_3)|^2\geq |\nabla m'|^2$ and $m_3^2\geq m_3^4=(1-|m'|^2)^2$.
Due to the similarity with vortex points in Ginzburg-Landau type functionals (see the seminal book of Bethuel, Brezis and H\'elein \cite{BBH-book}), the energetic cost of a micromagnetic vortex is given by 
\begin{align*}
  E_\delta(\text{Vortex})=2\pi{|\log \eps |}+O(1),
\end{align*}
(see e.g. \cite{IGNAT-XEDP}).

\medskip

{\nd \bf Regime.} We will focus on an energetic regime allowing for N\'eel walls, but excluding vortices. More precisely, we will assume that $\delta \to 0$ and $\eps=\eps(\delta)\to 0$ such that
\be \label{reg_eta_eps}
 \frac{1}{\delta|\log \delta|}=o\left( |\log \eps|\right)
\ee
and we will consider families of magnetization $\{m_\delta\}_{0<\delta<1/2}$ satisfying the energy bound
\begin{equation}
\label{regime:neel}
\sup_{\delta\to 0} \delta |\log \delta| E_\delta(m_\delta)<+\infty. 
\end{equation}
In particular, \eqref{reg_eta_eps} implies that the size $\eps$ of the vortex core is exponentially smaller than the size of the N\'eel wall core $\delta$, i.e., $\eps=O( e^{-\frac{1}{\delta |\log \delta|}})$.

\bigskip

{\nd \bf Compactness of N\'eel walls.}
We first show that the energetic regime \eqref{regime:neel} is indeed favorable for the formation of N\'eel walls. We start by proving a compactness result for $\Ss^2$-magnetizations in the regime \eqref{reg_eta_eps} and \eqref{regime:neel} that is reminiscent to the compactness results of Ignat and Otto in \cite{IO1} and \cite{IO2}.

\begin{thm}
\label{thm_comp}
Let $\delta>0$ and $\eps(\delta)>0$ satisfy the regime \eqref{reg_eta_eps}. Let $m_\delta\in H^1_\loc(\Omega,\mathbb{S}^2)$ satisfy \eqref{1b} and \eqref{regime:neel}.
Then $\{m_\delta\}_{\delta\to 0}$ is relatively compact in $L_{\loc}^2(\Om)$ and any limit 
$m:\Omega\to \Ss^2$ satisfies the constraints \eqref{1b} and 
\[ |m'|=1, \quad m_3=0, \quad \nabla\cdot m'=0 \, \, \textrm{ in } \, {\mathcal{D}}'(\Omega). \]
 \end{thm}

\medskip
The proof of compactness is based on an argument of approximating
$\Ss^2$-magnetizations by $\Ss^1$-valued magnetizations having the same level of energy (see Theorem \ref{lem_approx}). Such an approximation is possible due to our regime \eqref{reg_eta_eps} and \eqref{regime:neel} that excludes existence of topological point defects.

\bigskip

{\nd \bf Optimality of the N\'eel wall.}
Our second result proves the optimality of the N\'eel wall, namely that the N\'eel wall is the unique asymptotic minimizer of $E_\delta$ over $\Ss^2$-magnetizations within the boundary condition \eqref{1b}. For every magnetization $m:\Om\to \mathbb{S}^2$, we associate the energy density $\mu_\delta(m)$ as a non-negative $x_2$-periodic measure on
$\Omega\times \RR$ via 
\be
 \int_{\Omega\times \RR } \zeta \, d
\mu_\delta(m):= \frac{2}{\pi} \delta |\log \delta| \bigg(\int_{\Omega} \zeta(x,0)
\big(|\nabla m|^2+\frac 1 {\eps^2} m_{3}^2\big)\, dx+\frac 1 \delta \int_{\Omega \times \RR} \zeta |h(m')|^2\, dxdz\bigg),  \label{energydef}
\ee
for every $\zeta=\zeta(x,z)\in C_c(\Omega \times \RR)$. Recall that $h(m')$ denotes the $x_2$-periodic stray-field associated to $m'$ via \eqref{eqh}. We now show that the straight walls \eqref{1e} are the unique minimizers of $E_\delta$ as $\delta \to 0$ in which case the energy density $\mu_\delta$ is concentrated on a straight line in $x_2$-direction.

\begin{thm}
\label{thm:optimality}
Let $\delta>0$ and $\eps(\delta)>0$ satisfy the regime \eqref{reg_eta_eps}. Let $m_\delta\in H^1_\loc(\Omega,\mathbb{S}^2)$ satisfy \eqref{1b} and
\be
\label{hypo}
\limsup_{\delta\to 0} \delta|\log \delta| E_\delta(m_\delta)\leq \frac{\pi}{2} (1-\ma)^2.
\ee
Then there exists a subsequence $\delta_n \to 0$ such that $m_{\delta_n}\to m^*$ in $L_{loc}^2(\Om)$ where $ m^*$ is a straight
wall given by 
\be \label{1e} 
m^*(x_1,x_2)= \begin{cases}
m_{-\infty} & \text{for } x_1 < x_1^*, \\
m_{+\infty} & \text{for } x_1 > x_1^*,
\end{cases} \quad \text{ for some } x_1^* \in [-1,1].
\ee
In this case we have the concentration of the measures defined at \eqref{energydef} on the jump line of $m^*$:
\[ \mu_{\delta_n}(m_{\delta_n}) \rightharpoonup (1-\ma)^2 \, {\mathcal H}^1\llcorner \{x_1^*\}\times \Tt \times\{0\} \quad\textrm{weakly  $^*$ in } \, {\mathcal M}(\Om\times \R). \]
\end{thm}

\medskip

The energy bound \eqref{hypo} is relevant for N\'eel walls (see e.g. \cite{I}). The similar result in the case of $\Ss^1$-valued magnetizations 
was previously proved by Ignat and Otto in \cite{IO1}. Theorem~\ref{thm:optimality} represents the extension of that result to the case of $\Ss^2$-valued magnetizations. 
An immediate consequence of Theorem \ref{thm:optimality} is the following lower bound of the energy $E_\delta$ within the boundary conditions \eqref{1b}.

\begin{cor}
Let $\delta>0$ and $\eps(\delta)>0$ satisfy the regime \eqref{reg_eta_eps}. Let $m_\delta\in H^1_\loc(\Omega,\mathbb{S}^2)$ satisfy \eqref{1b}. Then
\be
\label{low_est}
\liminf_{\delta \to 0} \delta |\log \delta| E_{\delta}(m_{\delta} )\geq \frac{\pi}{2} (1-\ma)^2.
\ee
\end{cor}

\bigskip

\subsection{Dynamics. The Landau-Lifshitz-Gilbert equation.} The dynamics in ferromagnetism is governed by a torque balance which gives rise
to a damped gyromagnetic precession of the magnetization around the 
effective field defined through the micromagnetic energy. 
The resulting system is the Landau-Lifshitz-Gilbert (LLG) equations which is neither a Hamiltonian system nor a gradient flow. 

\medskip

Let us present the setting of LLG equations. As the condition \eqref{1b}  is not preserved by the LLG flow, we will \emph{impose} the boundary conditions \eqref{1b} at each time $t\geq 0$, and look for solutions of LLG equations in the space domain $$x\in \omega:= (-1,1) \times \m T.$$ In order to define the micromagnetic energy and its gradient on $\omega$, we introduce the functional calculus derived from the Laplace operator on $\omega$ with Dirichlet boundary conditions. More precisely, for $f \in H^{-1}(\omega)$,  we define $ g := (-\Delta)^{-1} f$ as the solution of
\begin{gather} \label{def:Delta_dir}
\begin{cases}
 - \Delta g = f & \quad \text{in } \omega, \\
g (x_1,x_2) =0 & \quad \text{ on } \partial \omega, \textrm{ i.e., for } \, |x_1|=1, x_2\in \Tt.
\end{cases}
\end{gather}
Then $(-\Delta)^{-1}$ is a bounded operator $H^{-1}(\omega) \to H^1_0(\omega)$ and a compact self-adjoint operator $L^2(\omega) \to L^2(\omega)$. We can therefore construct a functional calculus based on it, and denote as usual $|\nabla|^{-2s}: = \big[ (-\Delta)^{-1} \big]^{s}$ for $s=1/2$ and $s=1/4$.

\medskip

The dynamics 
of the state of the thin ferromagnetic sample is described by the time-dependent magnetization
\[ m=m(t,x): [0,+\infty) \times \omega \to \mathbb{S}^2,\]
that solves the following equation (see \cite{Gilbert, Landau-Lifshitz}):
\begin{equation}
 \label{eq:LLG-simple} \tag{LLG$_0$}\partial_t m+\alpha m\times \dt m+ \beta  m\times \nabla \tilde{E}_{\delta}(m)=0\:  \text{ on } [0,\infty) \times \omega.
\end{equation}
Here, $\times$ denotes the cross product in $\R^3$, while
$\alpha>0$ is the Gilbert damping factor characterizing the dissipation form of \eqref{eq:LLG-simple} and $\beta>0$ is the gyromagnetic ratio characterizing the precession. The micromagnetic energy $\tilde{E}_\delta$ corresponding to the domain $\omega$ is defined via \eqref{def:Delta_dir}:
\be \label{definitia_bis} 
\tilde{E}_\delta(m)=
\int_{\omega}\bigg( |\nabla m|^2  + \frac{1}{\eps^2} m_3^2\, \bigg)dx+ \frac 1{2 \delta}  \int_{\omega} \big||\nabla|^{-1/2}\nabla\cdot m'\big|^2 \, dx, \ee 
so that the gradient of the energy $\tilde E_\delta(m)$ is given as:
\begin{equation} \label{nablaE_bis}
\nabla \tilde E_{\delta}(m)=- 2 \Delta m+\bigg(\frac{1}{\delta} \p(m'),\frac{2m_{3}}{\eps^2}\bigg),
\end{equation}
where we have introduced\footnote{Observe that our original nonlocal operator appearing in the energy gradient \eqref{nablaE} can be written as
$(-\Delta)^{1/2} \q H(m') ={ -} \nabla |\nabla|^{-1} \nabla \cdot m'$.} the operator $\p$ acting on $m' \in H^1(\omega,\m R^2)$ via \eqref{def:Delta_dir}:
\[ \p (m') := - \nabla |\nabla|^{-1} \nabla \cdot m'. \]
Observe that as in \eqref{eqh2}, we have
\[ \int_\omega \big| |\nabla|^{ -1/2} \p(m') \big|^2 \, dx =  \int_{\omega} \big||\nabla|^{-1/2}\nabla\cdot m'\big|^2 \, dx. \]

\medskip

\begin{rem} \label{rm:1}
\begin{enumerate}
\item[i)]
We highlight that Theorems \ref{thm_comp} and \ref{thm:optimality} remain valid in the context of the micromagnetic energy $\tilde{E}_\delta$ on $\omega$ within the boundary conditions \eqref{1b}, i.e.,
$m(x_1, x_2)= m_{\pm\infty}$ for $x_1= \pm 1$ and every $x_2\in \Tt$.

\item[ii)] Note that for a map $m:\omega\to \Ss^2$, one has $\tilde E_\delta(m)<\infty$ if and only if $m\in H^1(\omega)$.
\end{enumerate}
\end{rem}

\medskip

In this paper, we consider a more general form of the Landau-Lifshitz-Gilbert equation including additional drift terms, which has been derived in a related setting in \cite{ZL, Th} (see also \cite{KMM}): 
\begin{equation}
 \label{eq:LLG} \tag{LLG}
\partial_t m+\alpha m\times \dt m+ \beta  m\times \nabla \tilde E_{\delta}(m)+(v\cdot \nabla )m= m\times (v\cdot \nabla)m \, \, \text{ on } [0,+\infty) \times \omega,
\end{equation}
where $v:[0,+\infty)\times \omega\to \R^2$ represents the direction of an applied spin-polarized current\footnote{By definition $(v\cdot \nabla) m=v_1\partial_1m+v_2\partial_2m$.}.

\medskip


 {\nd \bf Regime.} 
We analyze the dynamics of the magnetization through \eqref{eq:LLG} in the asymptotics $\delta \to 0, \eps(\delta)\to 0$ in the regime \eqref{reg_eta_eps}, while
\begin{equation}
\label{reg:alpha}\alpha=\nu \eps,\quad \beta=\lambda \eps,
\end{equation}
where $\nu>0$ is kept fixed and
\begin{equation}
\label{reg:lambda}
\lambda(\delta)=o\left(\sqrt{\delta |\log \delta|}\right).
\end{equation}
 
The dynamics of the magnetization for the equation \eqref{eq:LLG-simple} has been derived by Capella, Melcher and Otto \cite{CMA} (see also Melcher \cite{Melcher}) in the asymptotics $\eps\to 0$ with {\it fixed} $\delta$ (see \cite[Theorem 1]{CMA}). The more general equation \eqref{eq:LLG} ({\it in the absence of the non-local energy term}) was studied by Kurzke, Melcher and Moser in \cite{KMM} where they derived rigorously the motion law of point vortices in a different regime, namely $\eps \to 0$ and $\delta=+\infty$.  We highlight that in those papers, the parameter $\delta >0$ is kept {\it fixed or large} yielding a uniform $H^1$ bound via the energy; it is far beyond the grasp of \eqref{reg_eta_eps}. Therefore, in the analysis developed below, we will have to deal with the loss of the uniform $H^1$ bound; our strategy relies on the fine qualitative behavior of the magnetization presented in Theorems \ref{thm_comp} and \ref{thm:optimality} (that remain valid in the context of the micromagnetic energy $\tilde{E}_\delta$ on $\omega$ within the boundary conditions \eqref{1b}).

\medskip

In the present paper we consider initial data with finite energy at $\delta>0$ fixed.\footnote{Recall that in the regime \eqref{regime:neel} the initial energy blows up in the limit $\delta\to 0$.} We first have to solve the corresponding Cauchy problem for \eqref{eq:LLG} imposing the boundary conditions \eqref{1b} at each time $t\geq 0$. Naturally, we understand that here the boundary condition \eqref{1b} reads as
$$m(t, x_1, x_2)= m_{\pm\infty} \quad  \text{for } \quad x_1= \pm 1, \, x_2\in \Tt, \, t\geq 0.$$
Moreover these solutions have finite energy for all time $t\geq 0$. We insist on the fact that  the energy can possibly increase in time, unlike for \eqref{eq:LLG-simple} which is dissipative.

\begin{df}
\label{def:weak}
We say that $m$ is a global weak solution to \eqref{eq:LLG} if
\begin{gather} \label{def:wsol_reg}
 m \in L_{\mr{loc}}^\infty([0,+\infty), H^1(\omega, \m S^2)) \cap \dot H^1_{\loc}([0,+\infty), L^2(\omega)),
 \end{gather}
and $m$ solves the equation \eqref{eq:LLG} in the distributional sense $\q D'((0,+\infty) \times \omega)$.
\end{df}

Observe that the regularity assumption \eqref{def:wsol_reg} of this definition allows to make all terms in the \eqref{eq:LLG} meaningful in the distributional sense: this gives its relevance to the definition.  

 Indeed,  \eqref{def:wsol_reg} first gives (due to Remark \ref{rm:1} ii)) that $\tilde E_\delta(m(t))$ is finite for all $t \ge 0$. Also, $\nabla \tilde E_\delta(m)\in L^\infty_\loc([0,+\infty), H^{-1}(\omega))$ since for $\nabla m(t)\in L^2(\omega)$ then $\Delta m(t)\in H^{-1}(\omega)$, while $\p(m'(t))\in L^2(\omega)$. From there, we infer that $m\times \nabla \tilde E_\delta(m)\in L^\infty_\loc([0,+\infty), H^{-1}(\omega))$. Indeed, by setting
\begin{equation*}\begin{split}\langle m(t)\times \Delta m(t), \phi\rangle_{H^{-1}(\omega),H_0^{1}(\omega)}:=-\sum_{j=1}^2\int_{\omega} (m(t)\times \partial_j m(t))\cdot \partial_j \phi \,dx\end{split}\end{equation*} and by noticing that $\p(m')$ and $m_{3}$ belong to $L^\infty_\loc([0,+\infty),L^2(\omega))$, we get for $0<\eps\leq \delta$ small (see \eqref{estm_ned}):
\begin{equation*}
\|m(t)\times \nabla \tilde E_\delta(m(t))\|_{H^{-1}(\omega)}\leq \frac{C}{\eps}  \tilde E_\delta(m(t))^{1/2}.
\end{equation*}
All the other terms in \eqref{eq:LLG} belong to $L^2_\loc([0,+\infty)\times {\omega})$.

We construct global weak solutions for \eqref{eq:LLG} in the following theorem.

\begin{thm}\label{thm:cauchy}Let $\delta\in (0, 1/2)$ be fixed, $m^0 \in H^1(\omega, \m S^2)$ be an initial data and the spin current $v\in L^\infty([0,+\infty)\times \omega, \R^2)$. 

Then there exists a global weak solution $m$ to \eqref{eq:LLG} (in the sense of Definition \ref{def:weak}), which satisfies the boundary conditions
\begin{align}
\label{eq:bdc2}
\quad m(t=0, \cdot) & = m^0 \,  \quad \, \, \text{ in } \omega, \\
\quad
m(t,x_1,x_2) & = m^0(x_1,x_2) \quad  \text{ if } x_1 = \pm 1 \text{ and for every } \  x_2 \in \m T, \,  t \ge 0. \label{eq:bdc1}
\end{align}
Furthermore $m$ satisfies the following energy bound: for all $t \ge 0$,
\begin{equation}
\label{energy_eq}
\tilde{E}_{\delta} (m(t)) + \frac{\alpha}{2\beta} \int_0^t \| \partial_t m(s) \|^2_{L^2(\omega)} ds \le  \tilde{E}_{\delta} (m^0) \exp \left(  \frac{4}{\alpha \beta} \int_0^t  \| v (s) \|_{L^\infty(\omega)}^2 ds \right).
\end{equation}
\end{thm}

The proof of Theorem \ref{thm:cauchy} takes its roots in \cite{AlSo} via a space discretization. To the best of our knowledge however, there is no such result taking into account the non-local term in $\nabla \tilde{E}_\delta$ (see \eqref{nablaE_bis}). One needs to carry on the computations carefully, specially as it comes together with the constraint of $\m S^2$-valued maps. For the convenience of the reader we provide a full proof in Section 5 below.

\bigskip

We next specify our set of assumptions for the dynamics in the asymptotics $\delta, \eps(\delta) \to 0$:
\begin{enumerate}
\item[(A1)]  The initial data 
$m_{\delta}^0\in H^1(\omega, \Ss^2)$ satisfy \eqref{1b} and $\sup_{\delta\to 0} \delta |\log \delta| \tilde E_\delta(m_\delta^0)<+\infty$. 

\item[(A2)] The regime \eqref{reg_eta_eps} holds as $\delta\to 0$ and the parameters $\alpha$ and $\beta$ satisfy \eqref{reg:alpha} and \eqref{reg:lambda}.

\item[(A3)] The spin-polarized current satisfies \be
\label{reg:gamma} \|\vd\|_{L^\infty([0,+\infty)\times \omega)}^2\leq  \alpha \beta.\ee
In particular, we have $v_{\delta} \to 0$  in $L^{\infty}([0,+\infty) \times \omega)$.
\end{enumerate}
\medskip

Due to the energy estimate \eqref{energy_eq}, the energetic regime in (A1) holds for all times $t \ge 0$ (with no uniformity in $t$ though). In particular, Theorem \ref{thm_comp} implies that for all $t>0$, the magnetizations $\{\md(t)\}_\delta$ admit a subsequence converging in $L^2(\omega)$ to a limiting magnetization $(m'(t),0)$ as $\delta\to 0$. Our main result is that the subsequence does not depend on $t$ and that the limiting configuration is stationary.

\begin{thm} \label{thm:dyn} Let $\{m_{\delta}^0\}_{0<\delta<1/2}$ be a family of initial data in $H^1(\omega, \Ss^2)$. Suppose that the assumptions (A1), (A2) and (A3) above are satisfied. Let $\{m_\delta\}_{0<\delta<1/2}$ denote any family of global weak solutions to \eqref{eq:LLG} satisfying \eqref{eq:bdc2}, \eqref{eq:bdc1} and the energy estimate \eqref{energy_eq}. 

Then there exists a subsequence $\delta_n\to 0$ such that $m_{\delta_n}(t)\to m(t)$ in $L^2(\omega)$ for all $t\in [0,+\infty)$ as $n\to \infty$ where the accumulation point $m=(m',0)\in C([0,+\infty),L^2(\omega, \Ss^2))$ satisfies
\[ |m'(t)|=1,\quad \nabla\cdot m'(t)=0\quad \text{in }\mathcal{D}'( \omega),\quad \forall t\in [0,+\infty). \]
Moreover, the limit $m$ is stationary, i.e., 
\[ \partial_t m'=0\quad \text{in } \:\mathcal{D}'([0,+\infty)\times \omega). \]

\end{thm}
In particular, it follows immediately from Theorems \ref{thm:optimality} and \ref{thm:dyn} (and Remark \ref{rm:1} i)) that for well-prepared initial data the asymptotic magnetization is a static straight wall for all $t\geq 0$:
\begin{cor}Under the same assumptions as in Theorem \ref{thm:dyn}, 
assume moreover that the initial data are well-prepared:
\begin{equation*}
\limsup_{\delta \to 0}  \delta|\log \delta|\tilde{E}_{\delta}(m_{\delta}^0) \le \frac{\pi}{2}(1-m_{1,\infty})^2.
\end{equation*}
Let $\delta_n\to 0$ and let $x_1^\ast\in [-1,1]$ be such that $m_{\delta_n}^0\to m^*$ in $L^2(\om)$, where $ m^*$ is a straight wall defined by \eqref{1e}. Then we have $m_{\delta_n}(t)\to m^*$ in $L^2(\om)$ for all $t\geq 0$.

\end{cor}

\medskip

The paper is organized as follows. In Sections 2 and 3, we focus on the stationary results and prove Theorem \ref{thm_comp} and Theorem \ref{thm:optimality}. In Section 4, we prove Theorem \ref{thm:dyn}, assuming Theorem \ref{thm:cauchy}, which is proved in Section 5.
Finally, we prove in the Appendix a uniform estimate in the context of the Ginzburg-Landau energy, which is needed in the proof of Theorem \ref{thm_comp}.

\bigskip

In all the following $C$ will denote an absolute constant (independent of the parameters of the system) which can possibly change from one line to another.

\section{Approximation and compactness}

This section is devoted to the proof of Theorem \ref{thm_comp}. 
A similar compactness result to Theorem \ref{thm_comp} has been already established by Ignat and Otto in \cite[Theorem 4]{IO1} for $\Ss^1$-valued magnetizations. In order to establish compactness for $\Ss^2$-valued magnetizations we will use an argument consisting in approximating $\Ss^2$-valued maps by $\Ss^1$-valued maps with quantitative bounds given in terms of the energy, which is stated as follows.

\begin{thm}
\label{lem_approx}
Let $\beta\in (0,1)$. Let $\delta>0$ and $\eps(\delta)>0$ satisfy the regime \eqref{reg_eta_eps}, i.e.,
$$\frac{1}{\delta |\log \delta| |\log \eps|}\to 0 \quad  \textrm{as }\, \delta \to 0,$$ and let $m_\delta
=(m_\delta', m_{3, \delta})\in H^1_{\loc}(\Om, \Ss^2)$ satisfy \eqref{1b} and \eqref{regime:neel}. Then there exists $M_\delta\in H^1_{\loc}(\Om, \Ss^1)$ that satisfies \eqref{1b} such that
 \be
\label{cond-approx}  \int_{\Om} |M_\delta-m_\delta'|^2\, dx\leq C
\eps^{2\beta} E_\delta(m_\delta)\quad \textrm{and}\quad \int_{\Om} |\nabla (M_\delta-m_\delta')|^2\, dx\leq C E_\delta(m_\delta)\ee
and
\be
\label{amel2}
\int_{\Om\times \R} |h(M_\delta)-h(m'_\delta)|^2\, dxdz\leq C \eps^\beta E_\delta(m_\delta),
\ee
and
\be
\label{main_ineg}
 E_\delta(M_\delta)\leq  E_\delta(m_\delta) \left(1+o(1)\right),\ee
 where $o(1)=O\bigg(\big(\frac{1}{\delta |\log \delta| |\log \eps|}\big)^{\frac16-}\bigg)$ and $\frac 1 6-$ is any fixed positive number less than $\frac 1 6$.
Moreover, for every full square $T(x,r)$ centered at $x$ of side of length $2r$ with $\eps^{\beta}/ r\to 0$ as $\delta \to 0$, we have\footnote{In \eqref{amel1}, $o(1)$ is the same as in \eqref{main_ineg}.}
\be
\label{amel1}
\int_{T(x,r-2\eps^{\beta})} |\nabla M_\delta|^2\, dx\leq (1+o(1))\int_{T(x,r)}
\bigg(|\nabla m'_\delta|^2+\frac{1}{\eps^2}m_{3,\delta}^2\bigg)\, dx.
\ee

  \end{thm}

Theorem \ref{lem_approx} is reminiscent of the argument developed by Ignat and Otto \cite{IO2} with a major improvement given by \eqref{main_ineg}, i.e., the approximating $\Ss^1$-map $M_\delta$ has lower energy than the $\Ss^2$-map $m_\delta$ (up to $o(1)$ error).

\begin{proof}
To simplify notation, we will often omit the index $\delta$ in the following. We introduce a Ginzburg-Landau type energy density: 
\be \label{GL-en} \den(m')=|\nabla m'|^2+\frac{1}{\eps^2} (1-|m'|^2)^2. \ee
The approximation scheme is inspired by \cite{IO2}.

\medskip

\emph{Step 1. Construction of a squared grid.} For each shift
$t\in (0, \eps^{\beta})$, we consider the set
\[ H_t=\{x=(x_1, x_2)\in \RR\times (0,1)\, : \, x_2\in (\eps^\beta,1-\eps^\beta), \, x_2\equiv t \,\, (\text{mod }\eps^{\beta})\} \]
 and we repeat it $1$-periodically in $x_2$ to obtain a net of horizontal lines at a distance
$\eps^{\beta}$ in $\Omega$. By the mean value theorem, 
there exists $t\in (0,\eps^\beta)$
such that
\[ \int_{H_t} \den(m')\, d\h^1\leq \frac{1}{\eps^\beta} \int_{\Om} \den(m')\, dx.\]
If one repeats the above argument for the net of vertical
lines at distance $\eps^{\beta}$ in $\Om$, we get a shift $s\in (0, \eps^\beta)$ such that
the net  
\[ V_s:=\{x\in \Om\, :\, x_1\in (-1+\eps^\beta, 1-\eps^\beta), \, x_1\equiv s \,\, (\text{mod }
\eps^{\beta})\} \]
satisfies
\[ \int_{V_s} \den(m')\, d\h^1\leq \frac{1}{\eps^\beta} \int_{\Om} \den(m')\, dx. \] 
Set $\tilde V_s:=V_s\cup \{(x_1, x_2)\, :\, x_1\in\{\pm 1\}, x_2\in [0,1)\}$
and remark that $\int_{V_s} \den(m')\, d\h^1=\int_{\tilde V_s} \den(m')\, d\h^1$ since $m$ satisfies \eqref{1b}.
Therefore, we obtain an $x_2$-periodic squared grid
${\mathcal R}=H_t \cup \tilde V_s$ of size more than $\eps^\beta$ such that \be
\label{condR}
 \int_{{\mathcal R}} \den(m')\,
d\h^1 \leq \frac{2}{\eps^\beta} \int_{\Om} \den(m')\,
dx\leq \frac{2 E_\delta(m)}{\eps^\beta}\leq  \frac{C}{\eps^\beta \delta|\log \delta|}.\ee
Due to periodicity, one may assume that $\mathcal R$ includes the horizontal line $\R\times\{0\}$.

\medskip

\emph{Step 2. Vanishing degree on the cells of the grid $\mathcal R$.}
In order to approximate $m'$ in $\Om$ by $\m \Ss^1$-valued vector fields, it is
necessary for $m'$ to have zero degree on each cell of the
grid ${\mathcal R}$. Let us prove this property. For that, let $\mathcal C$ be a full squared cell of $\mathcal R$ having all four sides of the cell of length $\in [\eps^\beta, 4\eps^\beta]$. 
We know that \eqref{condR} holds (in particular, for $\den$ on $\mathcal C$). 
Set $\ds \ka:=\frac{1}{\delta|\log \delta|}=o\left( |\log \eps|\right)$. By Theorem~\ref{thm_main_GL} given in the Appendix, we deduce that $|m'|\geq 1/2$ on $\mathcal R$ and $\degr(m', \partial{\mathcal C})=0$ for small $\eps>0$.

\medskip

\emph{Step 3. Construction of an approximating $\m \Ss^1$-valued vector field $M$ of $m'$}. 
On each \textbf{full} squared cell ${\mathcal C}$ of $\mathcal R$ of side of length of order $\eps^\beta$, we
define
$u=u_\delta \in H^1({\mathcal C},\R^2)$ to be a  minimizer of 
\[ \min \bigg\{ \int_{\mathcal C} \den(u)\, dx\,:\, u=m'\, \textrm{ on }\,  \partial {\mathcal C}\bigg\}.  \]
Putting together all the cells, $u$ is now defined in the whole $\hul(\mathcal R)$ (which is $[-1,1]\times \m T$)  and satisfies \eqref{1b}.  Extend $u$ by $m^\pm$ for $\pm x_1\geq 1$ so that $u$ is defined now in $\Omega$ and is periodic in $x_2$. Moreover, by construction, 
\[ \int_{\Om} \den(u)\, dx\leq \int_{\Om} \den(m')\, dx. \]
By 
Theorem~\ref{thm_main_GL} given in the Appendix, we have 
\[ \eta:=\sup_{\Omega}\big||u|^2-1\big|\leq C\bigg(\frac{1}{\delta |\log \delta| |\log \eps|} \bigg)^{\frac16-}=o(1). \] 
In particular,
\[ |u|^2\geq 1-\eta\quad \textrm{ in }\quad \Om. \]
Therefore we define $M \in H^1(\Omega, \m \Ss^1)$ by
\[ M:=\frac{u}{|u|} \quad \textrm{in} \quad \Om. \]
So, $M$ satisfies \eqref{1b}.
We deduce that $|\nabla u|^2\geq |u|^2 |\nabla M|^2\geq (1-\eta) |\nabla M|^2$ in $\Om$ and
\begin{align}
\label{esti_harmon}
(1-\eta)\int_{\Om} |\nabla M|^2\, dx  \leq \int_{\Om} |\nabla u|^2\, dx &\leq \int_{\Om} \den(u)\, dx \leq \int_{\Om} \den(m')\, dx \leq  E_\delta(m).
\end{align}
We prove now \eqref{amel1} which a local version of \eqref{esti_harmon}. Using the above constructed grid, we cover $T(x, r-2\eps^\beta)\cap ([-1,1]\times \m T)$ by a subgrid $\cup_{k\in K} {\mathcal C}_k$, with $K$ finite, of \textbf{full} cells of $\mathcal R$ such that $\cup_{k\in K} {\mathcal C}_k\subset T(x, r)\cap  ([-1,1]\times \m T)$. Therefore we have:
\begin{align*}
\MoveEqLeft (1-\eta)\int_{T(x, r-2\eps^\beta)} |\nabla M|^2\, dx = (1-\eta)\int_{T(x, r-2\eps^\beta)\cap  ([-1,1]\times \m T)} |\nabla M|^2\, dx\\
&\le (1-\eta)\int_{\cup_{k\in K} {\mathcal C}_k} |\nabla M|^2\, dx \le \int_{\cup_{k\in K} {\mathcal C}_k} |\nabla u|^2\, dx \\
& \leq \int_{\cup_{k\in K} {\mathcal C}_k} \den(u)\, dx  \leq \int_{\cup_{k\in K} {\mathcal C}_k} \den(m')\, dx  \leq \int_{T(x, r)} \den(m')\, dx.
\end{align*}

\medskip

The goal is now to prove that the $\m \Ss^1$-valued vector field $M$ approximates
$m'$ in $L^2(\Om, \RR^2)$ and the $\dot{H}^1$-seminorm of $M$ is comparable with the one of $m'$.

\medskip

\emph{Step 4. Estimate $\|\nabla(M-m')\|_{L^2(\Om)}$.} Indeed, by \eqref{esti_harmon}, we have:
$$(1-\eta) \int_{\Om} |\nabla M|^2\, dx\leq E_\delta(m)\quad \textrm{and} \quad \int_{\Om}  |\nabla m|^2\, dx\leq  E_\delta(m). $$ Thus, the second estimate in
\eqref{cond-approx} holds.

\medskip

\emph{Step 5. Estimate $\|M-m'\|_{L^2(\Om)}$.}
By
Poincar\'e's inequality, we have for each \textbf{full} cell ${\mathcal C}$ of
${\mathcal R}$: \be \label{eve1} \int_{ {\mathcal C}}
\bigg| M-\fint_{\partial {\mathcal C}} M
\bigg|^2\, dx\leq C \eps^{2\beta} \int_{{\mathcal C}} |\nabla
M|^2\, dx \ee 
and \be
\label{eve2} \int_{ {\mathcal C}} \bigg|m'-\fint_{\partial{\mathcal
C}} m' \bigg|^2\, dx\leq C
\eps^{2\beta} \int_{{\mathcal C}} |\nabla m'|^2\, dx. \ee Writing $m'=\rho v'$ with $\rho\geq \frac 1 2$ on ${\mathcal R}$ (by Theorem \ref{thm_main_GL} in Appendix), we have $v'=M$ on ${\mathcal R}$ and by Jensen's inequality, we
also compute
\begin{align}
\MoveEqLeft \int_{ {\mathcal C}} \bigg| \fint_{ \partial {\mathcal
C}} (M-m') \bigg|^2\, dx = \int_{{\mathcal C}} 
\bigg| \fint_{ \partial {\mathcal C}} (v'-m') \bigg|^2\, dx=\h^2({\mathcal C}) \fint_{ \partial{\mathcal C}} (1-\rho)^2\, d\h^1 \nonumber \\
&  \le C \eps^{\beta} \int_{\partial {\mathcal C}}(1-\rho^2)^2\, d\h^1 \leq  C \eps^{\beta+2} \int_{\partial {\mathcal C}}
\den(m')\, d\h^1. \label{eve3}
\end{align}
Summing up \eqref{eve1}, \eqref{eve2} and \eqref{eve3} over all
the cells ${\mathcal C}$ of the grid ${\mathcal R}$, by \eqref{condR} and \eqref{esti_harmon}, we
obtain
\[  \int_{\Omega} |M-m'|^2\, dx'\leq C\eps^{2\beta} E_\delta(m). \]

\medskip

\emph{Step 6. Proof of \eqref{amel2}.} Let $h(m')=\nabla U(m')$ and $h(M)=\nabla U(M)$ be the unique minimal stray fields given by \eqref{eqh}.
By uniqueness and linearity of the stray field, we deduce that $h(m')-h(M)$ is the minimal stray field associated to $m'-M$, i.e.,
\[ h(m')-h(M)=h(m'-M). \]
Therefore, we have by interpolation:
\begin{align*}
\MoveEqLeft \int_{\Om\times \R} |h(M)-h(m')|^2 \, dx dz \stackrel{\eqref{eqh2}}{=}\frac 1 2 \int_{\Omega}
\left|\,|\nabla|^{-1/2}\nabla\cdot (M-m')\right|^2\, dx\\
&\le C  \int_{\Omega} \left|\,|\nabla|^{1/2}(M-m')\right|^2\, dx\\
&\le C  \bigg(\int_{\Omega} |M-m'|^2\, dx\bigg)^{1/2}  \bigg(\int_{\Omega} |\nabla(M-m')|^2\, dx\bigg)^{1/2}\\
&\stackrel{\eqref{cond-approx}}{\le} C \eps^\beta E_\delta(m).
\end{align*}

\medskip

\emph{Step 7. End of the proof.} It remains to prove \eqref{main_ineg}. Indeed, by \eqref{esti_harmon} and Step 6, we have:
\begin{align*}
E_\delta(M)&=\int_{\Om} |\nabla M|^2\, dx+ \frac 1 \delta \int_{\Om\times \R} |h(M)|^2 \, dx dz\\
&\leq \frac{1}{1-\eta}\int_{\Om} \den(m')\, dx +\frac 1 \delta \int_{\Om\times \R} |h(m')|^2 \, dx dz+C\big(\frac{\eps^\beta}{\delta}\big)^{1/2} E_\delta(m)
\\
&\leq (1+C\eta) E_\delta (m)
\end{align*}
because $\ds \big(\frac{\eps^\beta}{\delta}\big)^{1/2}\leq\eta$ by \eqref{reg_eta_eps}.
\end{proof}

\medskip Observe that Theorem \ref{lem_approx} remains true in the context of the energy $\tilde E_\delta$ on the domain $\omega$.

\begin{proof}[Proof of Theorem \ref{thm_comp}]
It is a direct consequence of the approximation result in Theorem \ref{lem_approx} and of the compactness result in \cite{IO1} (see Theorem 4 in \cite{IO1}, and also Theorem 2 in \cite{IO2}).
\end{proof}

\section{Optimality of the N\'eel wall}

We present now the proof of Theorem \ref{thm:optimality}. The similar result in the case of $\Ss^1$-valued magnetizations 
was proved by Ignat and Otto in \cite{IO1} (see Theorem 1 in \cite{IO1}). Theorem \ref{thm:optimality} represents the extension to the case of $\Ss^2$-valued magnetizations.

\begin{proof}[Proof of Theorem \ref{thm:optimality}]
 Let $M_\delta$ be the approximating $\Ss^1$-map of $m_\delta$ constructed in Theorem~\ref{lem_approx}. By \eqref{hypo} and \eqref{main_ineg}, we deduce that 
 \[ \limsup_{\delta\to 0} \delta|\log \delta| E_\delta(M_\delta)\leq \frac{\pi}{2} (1-\ma)^2. \]
Then Theorem 1 in \cite{IO1} implies the existence of a sequence $\delta=\delta_n$ and $x_1^*\in [-1,1]$ such that 
\[ M_\delta-m^*\to 0\quad \textrm{in} \quad L^2(\Om), \]
which by \eqref{cond-approx} entails $m_\delta- m^*\to 0$ in $L^2(\Om)$.  Moreover, the $x_2$-periodic uniformly bounded sequence of measures $\mu_\delta(M_\delta)$ has the property that $$\mu_\delta(M_\delta) \rightharpoonup \mu_0 \, \, *\textrm{-weakly in } \, {\mathcal M}(\Omega\times \R),$$ where $\mu_0$ is a non-negative $x_2$-periodic measure in $\Omega\times \R$. Our first aim is to prove that
\be 
\label{form_conv_mes}
 \mu_0=(1-\ma)^2\, {\mathcal H}^1\llcorner \{x_1^*\}\times \mathbb{T}\times\{0\} .\ee
Indeed, let us define the function $\chi:\Omega\to \R$ by
\[ \chi= \pm \frac 1 2  \quad \textrm{if } \, \pm x_1\geq \pm x_1^*.  \]
Then, by Step 3 of the proof of Theorem 1 (and Remark 4) in \cite{IO1}, it follows that
\[ \frac{1}{4\alpha}\int_{\Omega\times\RR}\zeta\,
d\mu_0=\int_{\Omega}\nabla\zeta \cdot m^* \chi \,
dx+(1-\alpha)\int_{\Omega}\zeta|D\chi| \]
for every $\alpha\in (0,1)$ and for every smooth test function $\zeta:\RR^3\to \RR$ which is $1-$periodic
in $x_2$ with compact support in $x_1$ and $x_3$. Then we compute 
\[ \int_{\Omega}\nabla\zeta \cdot m^* \chi \,
dx=-\ma \int_{0}^1\zeta(x_1^*, x_2, 0) \, dx_2\]
so that by setting $\ds \alpha:=\frac{1-\ma}{2}$, we conclude that
\[ \int_{\Omega\times\RR}\zeta\, d\mu_0=4\alpha^2 \int_{\{x_1^*\}\times [0,1)\times \{0\}}\zeta  \, d\h^1, \]
i.e., $\mu_0=(1-\ma)^2 {\mathcal H}^1\llcorner \{x_1^*\}\times \mathbb{T}\times\{0\}$.

\bigskip

It remains to show that $\mu_\delta(m_\delta)\rightharpoonup \mu_0$ in ${\mathcal M}(\Omega\times \R)$. Indeed, by \eqref{hypo}, 
there exists a $x_2$-periodic nonnegative measure $\mu\in {\mathcal M}(\Omega\times \R)$ such that up to a subsequence, 
\be
\label{conv_mes}
\mu_\delta(m_\delta)
 \rightharpoonup \mu \, \, \textrm{weakly $^\ast$ in } \, {\mathcal M}(\Omega\times \R).\ee
The aim is to show that 
\[ \mu=\mu_0. \]
Indeed, let $r>0$ and $x=(x_1^*, x_2)\in \Om$ with $x_2\in [0,1)$. We consider an arbitrary smooth nonnegative test function $\zeta:\R^3\to [0,+\infty)$ that is $x_2$-periodic with compact support in $x_1$ and $x_3$ such that $\zeta\equiv 1$ on $T(x,r)\times (-\gamma,\gamma)$ for some fixed $\gamma>0$ (recall that $T(x,r)$ is the full closed square centered at $x$ of side of length $2r$). Within the notation \eqref{GL-en}, by Theorem \ref{lem_approx}, we have for $\beta=1/2$ and $\ds \eta=\bigg(\frac{1}{\delta |\log \delta| |\log \eps|} \bigg)^{\frac16-}$:
\[ \delta |\log \delta|\int_{T(x,r-2\eps^\beta)}|\nabla M_\delta|^2\stackrel{\eqref{amel1}}{\leq} (1+C\eta) \delta |\log \delta| \int_{\Om} \den(m_\delta')\zeta(x,0)\, dx \]
and
\begin{align*}
 |\log \delta| \int_{T(x,r-2\eps^\beta)\times (-\gamma, \gamma)}&|h(M_\delta)|^2\, dx dz\leq  |\log \delta| \int_{\Om\times \R}|h(M_\delta)|^2\zeta(x,z)\, dx dz \\
& \stackrel{\eqref{amel2}}{\leq} |\log \delta| \int_{\Om\times \R}|h(m'_\delta)|^2\zeta(x,z)\, dx dz+\|\zeta\|_{L^\infty} O(\delta {\eps^\beta})^{1/2} |\log \delta| E_\delta(m_\delta).
\end{align*}
Therefore, by \eqref{reg_eta_eps}, we obtain:
$$
\liminf_{\delta\to 0}\int_{T(x,r-2\eps^\beta)\times (-\gamma, \gamma)}\, d\mu_\delta(M_\delta) 
\leq \liminf_{\delta\to 0} \int_{\Om\times \R}
\zeta(x,z)\, d\mu_\delta(m_\delta)\stackrel{\eqref{conv_mes}}{=} \int_{\Om\times \R}
\zeta(x,z)\, d\mu.
$$
On the other hand, by \eqref{form_conv_mes}, one has 
\[ 2r(1-\ma)^2=\mu_0(\dot{T}(x,r)\times\{0\})\leq   
\liminf_{\delta\to 0}\int_{T(x,r-2\eps^\beta)\times (-\gamma, \gamma)}\, d\mu_\delta(M_\delta), \]
where $\dot{T}(x,r)$ is the interior of $T(x,r)$. Thus, we conclude
\begin{align*}
\MoveEqLeft 2r(1-\ma)^2\leq  \int_{\Om\times \R}
\zeta(x,z)\, d\mu.
\end{align*}
Taking infimum over all test functions $\zeta$ and then infimum over $\gamma\to 0$, we deduce
\[ 2r(1-\ma)^2\leq \mu({T}(x,r)\times\{0\}).\]
Setting 
\[ \mathrm{Line} :=\{x_1^*\}\times \Tt \times\{0\}\quad \textrm{and} \quad \mu_{L}:=\mu\, \llcorner \, \mathrm{Line}, \]
we deduce that $\mu_L(S)\geq (1-\ma)^2\h^1(S)=\mu_0(S)$ for every (closed) segment $S\subset \mathrm{Line}$; therefore, $\mu_0\leq \mu_L\leq \mu$ as measures in ${\mathcal M}(\Omega\times \R)$. In particular, 
$$
\mu_0(\mathrm{Line})  \leq \mu_L(\mathrm{Line})\leq \mu(\Om\times \R)\leq \liminf_{\delta\to 0}  \int_{\Om\times \R} d\mu_{\delta}(m_\delta)\stackrel{\eqref{hypo}}{\leq}\mu_0(\mathrm{Line}),
$$
thus, 
\[ \mu=\mu_L=\mu_0 \textrm{ in } \, {\mathcal M}(\Omega\times \R).\]
Now \eqref{low_est} is straightforward.
\end{proof}

\section{Asymptotics of the Landau-Lifschitz-Gilbert equation}

We start now the study of the dynamics of the magnetization. We assume Theorem \ref{thm:cauchy} holds and postpone its proof to the next Section; our goal here is to establish Theorem \ref{thm:dyn}. 
Let $\{m_\delta^0\}_{0<\delta<1/2}$ be a family of initial data as in Theorem \ref{thm:dyn} and let 
\[ m_\delta=(m'_\delta, m_{3,\delta}):[0,+\infty)\times \omega\to \Ss^2 \]
be any family of global weak solutions to \eqref{eq:LLG} satisfying \eqref{eq:bdc2}, \eqref{eq:bdc1} and the energy estimate \eqref{energy_eq}.
Throughout this section we assume that (A1), (A2) and (A3) are satisfied.

\medskip

Let us also recall the energy inequality, on which we will crucially rely:
\begin{equation}
\tag{\ref{energy_eq}}
\tilde E_{\delta} (\md(t)) + \frac{\alpha}{2\beta} \int_0^t \| \partial_t \md(s) \|^2_{L^2(\omega)} ds \le  \tilde{E}_{\delta} (m_\delta^0) \exp \left(\frac{4}{\alpha \beta} \int_0^t  \| v_\delta (s) \|_{L^\infty(\omega)}^2 ds \right).
\end{equation}
In particular, it follows from \eqref{energy_eq} and the assumption (A3) on $v_\delta$ that
\begin{equation}
\label{energy-engin}
\tilde E_{\delta} (\md(t)) + \frac{\nu}{2\lambda} \int_0^t \| \partial_t \md(s) \|^2_{L^2(\omega)} ds \le  \tilde E_{\delta} (m_\delta^0) \exp \left(CT \right), \, \, 0<t\leq T,
\end{equation}
and therefore it follows from the energy bound (A1) on the initial data that
\begin{equation}\label{ineq:energy-t}
\forall T>0, \quad \sup_{0<\delta<1/2} \delta |\log \delta |\tilde E_\delta(m_\delta(T))<+\infty.
\end{equation}
Also, we infer the following bound on the time derivative  in $L^2_\loc([0,+\infty)\times \omega)$:
\begin{equation}\label{ineq:energy-time}
\forall T>0, \quad \|\partial_t \md\|_{L^2([0,T], L^2(\omega))}\leq C\exp(CT)\frac{\sqrt{\lambda}}
{\sqrt{\delta |\log(\delta)|}}.
\end{equation}
This is however not a uniform bound on $\lambda/(\delta |\log(\delta)|)$ as $\delta \to 0$ in the regime \eqref{reg:lambda}. Nevertheless, in the next proposition, we will establish a uniform bound of $\{\partial_t \md\}$ in the weaker space $L^2_\loc(H^{-1})$:

\begin{pro}\label{prop:equicont}Under the assumptions of Theorem \ref{thm:dyn},  we have for all $T>0$:
\begin{equation*}
\|\partial_t \md\|_{L^2([0,T], H^{-1}(\omega))}\leq \frac{C(T)}{\sqrt{\delta |\log(\delta)|}} \left( \lambda + \frac{ \lambda \e}{\delta} + \e^2  \right).
\end{equation*}
\end{pro}

\begin{proof}Let $T>0$.
By \eqref{eq:LLG} we have on $[0,+\infty)\times \omega$:
\begin{equation*}
\partial_t \md=-\alpha \md\times \dt \md- \beta  \md\times \nabla \tilde E_{\delta}(\md)-(v_\delta\cdot \nabla )\md+ \md\times (v_\delta\cdot \nabla)\md.
\end{equation*}
First, the inequality \eqref{ineq:energy-time} yields
\begin{equation*}
\begin{split}
\left\|  \alpha \md\times \dt \md  \right\|_{L^2([0,T],L^2(\omega))}&\leq C\exp(CT) \frac{\eps\sqrt{\lambda}}{\sqrt{ \delta |\log(\delta)|}}\\
&\leq C\exp(CT)\left(\frac{\eps^2}{ \sqrt{\delta |\log(\delta)|}}+\frac{\lambda}{ \sqrt{\delta |\log(\delta)|} }\right).
\end{split}
\end{equation*}
Next, by  (A3) we have
\begin{equation*}
\begin{split}
\left\| (v_\delta\cdot \nabla )\md   \right\|_{L^2([0,T],L^2(\omega))}&+\left\|  \md\times (v_\delta\cdot \nabla)\md \right\|_{L^2([0,T],L^2(\omega))}\\
&\leq C\sqrt{T}\exp(CT)\frac{\eps\sqrt{\lambda} }{\sqrt{\delta |\log(\delta)|}}\\
&\leq C\sqrt{T}\exp(CT) \left(\frac{\eps^2}{ \sqrt{\delta |\log(\delta)|}}+\frac{\lambda}{ \sqrt{\delta |\log(\delta)|} }\right).
\end{split}
\end{equation*}
Finally, recalling \eqref{nablaE_bis}, we have
\begin{align}
\nonumber
&\beta \left\| \md(t)\times  \nabla \tilde E_{\delta}(\md)(t)\right\|_{H^{-1}(\omega)}  \\
& \label{estm_ned}\leq
C \left(\lambda \eps \|\nabla \md (t)\|_{L^2(\omega)}+ \frac{\lambda \eps}{\delta} \|\nabla \md (t) \|_{L^2(\omega)}+ \lambda \left\|\frac{m_{3, \delta}(t)}{\eps}\right\|_{L^2(\omega)}\right)\\
\nonumber
&\leq C\exp(CT)\frac{\lambda}{\sqrt{\delta |\log(\delta)|}}  \left( 1+ \frac{\e}{\delta} \right).
\end{align}
Combining the previous estimates we obtain the estimate of the Proposition.
\end{proof}

\medskip

We now prove Theorem \ref{thm:dyn}:

\begin{proof}[Proof of Theorem \ref{thm:dyn}] Let $T>0$.
By Proposition \ref{prop:equicont} and  the assumptions \eqref{reg_eta_eps} and  \eqref{reg:lambda} on $\eps$, $\delta$ and $\lambda$, the family $\{\partial_t m_{\delta}\}_{0<\delta <1/2}$ is bounded in $L^2([0,T], H^{-1}(\omega))$. 
On the other hand, $\{m_{\delta}\}_{0<\delta <1/2}$ is bounded in $L^\infty([0,T], L^2(\omega))$. Therefore by Aubin-Lions Lemma (see e.g. Corollary 1 in \cite{Simon}) it is relatively compact in $\q C([0,T], H^{-1}(\omega))$.  Thus by a diagonal argument there exists $\delta_n\to 0$ and $m\in \q C([0,+\infty), H^{-1}(\omega))$ such that  $m_{\delta_n}\to m$ in  $\q C([0,T], H^{-1}(\omega))$ for all $T>0$ as $n\to \infty$.

On the other hand, let $t\in [0,+\infty)$. In view of the bound \eqref{ineq:energy-t} we conclude from Theorem \ref{thm_comp}
that any subsequence of $(m_{\delta_n}(t))_{n\in \mathbb{N}}$ is relatively compact in $L^2(\omega)$. Since $m_{\delta_n}(t)\to m(t)$ in $H^{-1}(\omega)$  we infer that the full sequence $m_{\delta_n}(t) \to m(t)=(m'(t),0)$ strongly in $L^2(\omega)$ as $n\to \infty$, where $|m'(t)|=1$, $m_3(t)=0$ almost everywhere and $\nabla\cdot m'(t)=0$ in the sense of distributions. In particular, $t\mapsto \|m(t)\|_{L^2(\omega)}=|\omega|^{1/2}\in \q C([0,+\infty),\R)$. 

Let us now prove that $m \in \q C([0,+\infty),L^2(\omega))$. Indeed, consider a sequence of times $t_n \ge 0$ converging to $t \ge 0$. As $m \in \q C([0,T], H^{-1}(\omega))$ and $m(t_n)$ is bounded in $L^2(\omega)$, we infer that $m(t_n) \tendf m(t)$ weakly in $L^2(\omega)$. But we just saw that $\| m(t_n) \|_{L^2(\omega)} \to \| m(t) \|_{L^2(\omega)}$, so that in fact $m(t_n) \to m(t)$ strongly in $L^2(\omega)$. This is the desired continuity.

Finally, Proposition \ref{prop:equicont} and  \eqref{reg:lambda} imply that 
\[ \partial_t m_{\delta_n}\to 0=(\partial_t m',0)\quad \text{in}\quad \mathcal{D}'([0,+\infty)\times \omega), \] which concludes the proof.
\end{proof}

\section{The Cauchy Problem for the Landau-Lischitz-Gilbert equation}

In this section we handle the Cauchy problem for the Landau-Lifshitz-Gilbert equation in the energy space.

\begin{proof}[Proof of Theorem \ref{thm:cauchy}]
We use an approximation scheme by discretizing in space. We first introduce some notation.

\bigskip

\emph{Notation and discrete calculus:}

\bigskip

Let $n \ge 1$ be an integer, $h = 1/n$ and $\omega_h = h \m Z^2  \cap \overline{\omega}$.
For a vector field $m^h : \omega_h \to \m R^3$, we will always assume $x_2$-periodicity in the following sense: 
\[ \forall x_1  \in h \m Z \cap [-1,1], \  e \in \m Z, \quad m^h (x_1, 1+ eh) = m^h (x_1, eh). \] 
We then define the differentiation operators as follows: for
$x = (x_1,x_2) \in \omega_h,$
\begin{align*}
\partial^h_1 m^h (x) & = \begin{dcases}
\frac{1}{2h} (  m^h (x_1 + h,x_2) - m^h (x_1 - h,x_2)) & \text{if } |x_1| <1, \\
\pm \frac{1}{2h} (m^h(x_1,x_2) - m^h(x_1 \mp h,x_2)) & \text{if } x_1 = \pm 1,
\end{dcases} \\
\partial^h_2 m^h (x) & =  \frac{1}{2h} (  m^h (x_1,x_2+h) - m^h (x_1, x_2 -h)).
\end{align*}
Observe that $\partial_1^h$ is the half sum of the usual operators $\partial_{1+}^h$ and $\partial_{1-}^h$ vanishing at the boundary $x_1 =  1$ and $x_1=-1$ respectively.
Also we define the discrete gradient and laplacian: denoting $(\hat e_1, \hat e_2)$ the canonical base of $\m R^2$, we let
\begin{gather*}
\nabla^{h}  m^h = \sum_{k=1}^2  \partial_k^{h} m^h \otimes  \hat e_k, \quad \Delta^h m^h  = \sum_{k=1}^2 \partial_k^{h} \partial_k^{h} m^h.
\end{gather*}
We introduce the scalar product
\[ \langle m^h , \tilde m^h \rangle_h = h^2 \sum_{x \in \omega_h} m^h (x) \cdot \tilde m^h (x) \]
and the $L^2_h$-norm and $\dot H^1_h$-seminorm:
\[ | m^h |_{L^2_h}^2 := \langle m^h , m^h \rangle_h, \quad |m^h |_{\dot H^1_h}^2 := \langle \nabla^{h} m^h, \nabla^{h} m^h \rangle_h. \]
Then we have the integration by parts formulas:
\begin{align*}
\langle \partial_1^{h} m^h, \tilde m^h \rangle_h & = -  \langle m^h ,\partial_1^{h} \tilde m^h \rangle_h + h \sum_{x  \in \omega_h, \,  x_1 =1}  m^h (x) \tilde m^h (x) -  h \sum_{x \in \omega_h, \,  x_1 = -1}  m^h (x) \tilde m^h(x), \\
\langle \partial_2^{h} m^h , \tilde m^h \rangle_h & = -  \langle m^h ,\partial_2^{h} \tilde m^h \rangle_h,
\end{align*}
where we used the above boundary conditions and periodicity.
 
We now define the sampling and interpolating operators $S^h$ and $I^h$.
We discretize a map $m: \omega \to \m R^3$ by defining $ S^h m : \omega_h \to \m R^3$ as follows:
\[ S^h m(x) = \begin{cases}
 \ds \frac{1}{h^2} \int_{C_x^h} m(y) dy & \text{if } x_1 <1, \\
 m(x) & \text{if } x_1=1,
 \end{cases}  \]
where $C_x^h = \{ y \in \omega \mid x_k \le y_k < x_k + h, \ k=1,2 \}$.
We will also identify $S^h m$ with the function $\overline\omega \to \m R^3$ which is constant on each cell $C_x^h$ for $x \in \omega_h$ with value $S^h m(x)$. With this convention, $S^h m$ is the orthogonal projection onto piecewise constant functions on each cell $C_x^h$ in $L^2(\omega)$. Also we have $| S^h m |_{L^2_h} = \| S^h m \|_{L^2(\omega)}$, and
\begin{gather} \label{S_cont}
| S^h m |_{L^2_h} \le \| m \|_{L^2(\omega)}, \quad | \nabla^h S^h m |_{L^2_h} \le \| \nabla m \|_{L^2(\omega)}, \quad | S^h m |_{L^\infty_h} \le \| m \|_{L^\infty(\omega)}.
\end{gather}
We interpolate a discrete map $m^h : \omega_h \to \m R^3$ to $I^h m^h: \omega \to \m R^3$
by a quadratic approximation as follows:
if $x \in C_y^h$ with $y \in \omega_h$, we set
\begin{gather*}
I^h m^h (x) = m^h(y) + \sum_{k=1}^2  \partial_k^{h+} m^h(y) (x_k - y_k) + \partial_1^{h+} \partial_2^{h+} m^h(y) (x_1 - y_1) (x_2 - y_2), \\
\text{where } \partial_k^{h+} m^h(y)  = \begin{dcases}
\frac{1}{h} (  m^h(y + h \hat e_k) - m^h (y)) & \text{if } k=2 \text{ or } (k=1 \text{ and } y_1 <1), \\
0 & \text{if } k=1 \text{ and } y_1 = 1.
\end{dcases}
\end{gather*}
One can check that $I^h m^h \in H^1(\omega)$ is continuous (it is linear in each variable $x_1$ and $x_2$, and coincide with $m^h$ at every point of $\omega_h$), quadratic on each cell $C_y^h$, and
\begin{gather} \label{eq:cont_dis} 
|m^h|_{L^2_h} \sim \| I^h m^h \|_{L^2(\omega)}, \quad  |\nabla^h m^h |_{L^2_h} \sim \| \nabla I^h m^h  \|_{L^2(\omega)}, \quad |m^h|_{L^\infty_h} = \| I^h m^h \|_{L^\infty(\omega)}
\end{gather} 
(we refer, for example, to \cite{MS98}).

We discretize the non-local operator $\q P$ so as to preserve the structure of a discrete form of $\ds \int | |\nabla|^{-1/2} \nabla \cdot m'|^2$. For this, notice that $|\nabla|^{-1}$ and $|\nabla|^{-1/2}$ naturally act as compact operators on $L^2(\omega)$, and hence if $m^h: \omega_h \to \m R^3$, $|\nabla|^{-1} m^h \in L^2(\omega)$ has a meaning. Also observe that due to Dirichlet boundary conditions, $d/dt$ commutes with $(-\Delta)^{-1}$, and hence with any operator of the functional calculus: in particular,
\[ \frac{d}{dt} | \nabla|^{-1} m =  | \nabla|^{-1} \frac{dm}{dt}. \]
Therefore we define for $m^h: \omega_h \to \m R^3$ the discrete operator:
\[ \q P^h m^h{}' := - \nabla^h S^h(|\nabla|^{-1} \nabla^h \cdot m^h{}'). \]
Then as $\| \nabla^h S^h m \|_{L^2_h} \le C \|\nabla m \|_{L^2(\omega)}$ we have
\begin{align} \label{P_cont}
\|  \q P^h m^h{}' \|_{L^2_h} \le C \| |\nabla|^{-1} \nabla^h \cdot m^h{}' \|_{\dot H^1(\omega)} \le C \| \nabla^h \cdot m^h{}' \|_{L^2(\omega)} \le C \| \nabla^h m^h{}' \|_{L^2_h}.
\end{align}

\emph{Step 1: Discretized solution and uniform energy estimate.}

Let 
\[ v^h(t) := S^h v(t) : \omega_h \to \m R^3 \quad \text{and}  \quad m^h_0(x) :=  \frac{1}{|S^h(m_0)(x)|} S^h (m^0)(x). \]
We consider the solution $m^h(t) : \omega_h  \to \m R^3$ to the following discrete ODE system: for $x= (x_1,x_2) \in \omega_h $ such that $|x_1| <1$, then
\begin{gather} \label{eq:mh}
\left\{ \begin{aligned}
&\frac{dm^h}{dt} + m^h \times \left( \alpha \frac{dm^h}{dt} + \beta \left( -2\Delta^h m^h + \left( \frac{1}{\delta} \q P^h(m^h{}'), \frac{2}{\e^2} m^h_3 \right) \right) \right. \\
& \qquad \left. \vphantom{\int} - (v^h \cdot \nabla^h) m^h -  m^h \times (v^h \cdot \nabla^h) m^h \right) =0 \\
& m^h(0,x) =  m^h_0(x),
\end{aligned} \right.
\end{gather}
and at the boundary
\begin{gather} \label{eq:mh_bd}
m^h(t,-1,x_2) = m^h_0(-1,x_2), \quad m^h(t,1,x_2) = m^h_0(1,x_2).
\end{gather}
As the operator $A(m^h) : \mu \mapsto \mu + \alpha m^h \times \mu$ is (linear and) invertible, this ODE takes the form
\[ \frac{dm^h}{dt} = A(m^h)^{-1}( \Phi(m^h)), \]
where 
\[ \Phi (m^h) = m^h \times \left( \beta \left( -2\Delta^h m^h + \left( \frac{1}{\delta} \q P^h(m^h{}'), \frac{2}{\e^2} m^h_3 \right) \right)  - (v^h \cdot \nabla^h) m^h - m^h \times (v^h \cdot \nabla^h) m^h \right) \]
is $\q C^\infty$. Hence the Cauchy-Lipschitz theorem applies and guarantees the existence of a maximal solution. Furthermore, we see that for all $x \in \omega_h$,
\begin{align*}
\frac{d}{dt} |m^h(t,x)|^2 & = 2 \left( m^h(t,x), \frac{d}{dt} m^h(t,x) \right) \\
& = \left( m^h(t,x),  m^h(t,x) \times \left( \alpha  \frac{d}{dt} m^h(t,x) + \Phi(m^h)(t,x) \right) \right) =0.
\end{align*}
This shows that for all $x \in \omega_h$, $| m^h (t,x)| = 1$ remains bounded, and hence $m^h$ is defined for all times $t \in \m R$.

We now derive an energy inequality for $m^h$. For this we take the $L^2_h$ scalar product of \eqref{eq:mh} with $\ds m^h \times \frac{dm^h}{dt}$. Recall that if $a, b, c \in \m R^3$, then $(a\times b) \times c = (a \cdot c) b - (a\cdot b)c$, hence 
\[ (c \times a) \cdot (c \times b) = ((c\times a) \times c) \cdot b) = (a \cdot b) |c|^2 - (c \cdot a) (c \cdot b) \]
 so that for any $\tilde m \in \m R^3$, and pointwise $(t,x) \in [0,+\infty) \times \omega_h$
\[  \left( m^h(t,x) \times \frac{dm^h}{dt} (t,x) \right) \cdot ( m^h(t,x)  \times \tilde m) = \frac{dm^h}{dt} (t,x) \cdot \tilde m. \]
Hence we have the pointwise equalities for $x \in \omega_h$ with $|x_1|<1$:
\begin{align*}
\left( m^h \times \frac{dm^h}{dt} \right) \cdot \left( m^h \times \alpha \frac{dm^h}{dt} \right) = \alpha \left| \frac{dm^h}{dt} \right|^2, \\
\left( m^h \times \frac{dm^h}{dt} \right) \cdot \left( m^h \times 2 ( 0, 0, m^h_3 )^T \right) = \frac{d}{dt} |m^h_3|^2.
\end{align*}
If $x \in \omega_h$ is on the boundary, that is $|x_1|=1$, then $\ds \frac{dm^h}{dt}(0,x) =0$ due to \eqref{eq:mh_bd}, and the previous identities also hold: we can therefore sum over $x \in \omega_h$.
Now consider the term involving the discrete Laplacian. The discrete integration by parts yields no boundary term due to $dm^h/dt$; and of course $d/dt$ commutes with $\nabla^h$. Therefore
\begin{align*}
\MoveEqLeft\left\langle m^h \times \frac{dm^h}{dt},   m^h \times (-2  \Delta^h m^h) \right\rangle_h = -2 \left\langle \frac{dm^h}{dt},     \Delta^h m^h \right\rangle_h = 2 \left\langle \nabla^h \frac{dm^h}{dt}, \nabla^h m^h \right\rangle_h \\
& = \frac{d}{dt} \left\| \nabla^h m^h \right\|^2_h.
\end{align*}

For the nonlocal term, $|\nabla|^{-1/2}$ is a self adjoint operator on $L^2$ due to the Dirichlet boundary conditions: the integration by parts yields no boundary term either. More precisely, as $d/dt$ commutes with all space operators,  we have
\begin{align*}
\MoveEqLeft \left\langle \frac{dm^h{}'}{dt} , \q P^h m^h{}' \right\rangle_{h}  = - \left\langle \nabla^h \frac{dm^h{}'}{dt} , S^h |\nabla|^{-1} \nabla^h \cdot m^h{}' \right\rangle_{h}  \\
&  = \int  \nabla^h \frac{dm^h{}'}{dt} \cdot (|\nabla|^{-1} \nabla^h \cdot m^h{}')  = - \int \left( |\nabla|^{-1/2} \nabla^h \cdot \frac{dm^h{}'}{dt} \right) (|\nabla|^{-1/2} \nabla^h \cdot m^h{}') \\
 & = - \frac{1}{2} \frac{d}{dt} \| |\nabla|^{-1/2} \nabla^h \cdot m^h{}' \|_{L^2}^2.
 \end{align*}
Thus we get
\begin{multline*}
 \alpha \left\| \frac{dm^h}{dt} \right\|_{L^2_h}^2 + \beta \frac{d}{dt} \left( \| \nabla^h m^h \|_{L^2_h}^2 +  \frac{1}{\delta} \| |\nabla|^{-1/2} (\nabla^h \cdot m^h{}') \|_{L^2}^2 + \frac{1}{\e^2} \| m_3^h \|_{L^2_h}^2 \right)  \\
 =  \left\langle  (v^h \cdot \nabla^h) m^h +  m^h \times (v^h \cdot \nabla^h) m^h, \frac{dm^h}{dt} \right\rangle_{L^2_h}.
 \end{multline*}
 Denote
\[ E^h(m^h) =  \| \nabla^h m^h \|_{L^2_h}^2 +  \frac{1}{2\delta} \| |\nabla|^{-1/2} (\nabla^h \cdot m^h{}') \|_{L^2}^2 + \frac{1}{\e^2} \| m_3^h \|_{L^2_h}^2. \]
Now we have
\begin{align*}
\MoveEqLeft
\left| \left\langle (v^h \cdot \nabla^h) m^h +  m^h \times (v^h \cdot \nabla^h) m^h, \frac{dm^h}{dt} \right\rangle_{L^2_h} \right| \\
& \le \sqrt {2} \| v^h \|_{L^\infty_h} \| \nabla^h m^h \|_{L^2} \left\| \frac{dm^h}{dt} \right\|_{L^2_h} \le \sqrt {2} \| v \|_{L^\infty} E^h(m^h)^{1/2}  \left\| \frac{dm^h}{dt} \right\|_{L^2_h} \\
& \le \frac{\alpha}{2} \left\| \frac{dm^h}{dt} \right\|_{L^2_h}^2 +  \frac{4}{\alpha}  \| v \|_{L^\infty}^2 E^h(m^h).
\end{align*}
Thus we obtained:
\[ \frac{\alpha}{2\beta}  \left\| \frac{dm^h}{dt} \right\|_{L^2_h}^2 + \frac{d}{dt} E^h(m^h) \le  \frac{4}{\alpha \beta}   \| v \|_{L^\infty}^2 E^h(m^h). \]
By Gronwall's inequality, we deduce
\begin{multline} \label{eq:energy_dis}
E^h(m^h(t)) + \frac{\alpha}{2 \beta} \int_0^t \left\| \frac{dm^h}{dt}(s) \right\|_{L^2_h}^2 ds
\le E^h(m^h(0)) \exp \left(  \frac{4}{\alpha \beta} \int_0^t \| v(s) \|_{L^\infty}^2 ds \right).
\end{multline}

\emph{Step 2: Continuous limit of the discretized solution}

Notice that $\| v^h \|_{L^\infty} \le \| v \|_{L^\infty}$. Also, as $m_0 \in H^1(\omega)$, then $ m^h(0) \to m_0$ in $H^1(\omega)$ and
\[ E^h(m^h(0)) \to \tilde{E}_\delta(m_0). \]

Fix $T>0$. It follows from \eqref{eq:energy_dis} and \eqref{eq:cont_dis} that the sequence $I^h m^h$ is bounded in $L^\infty([0,T],H^1(\omega))$ and in $\dot H^1([0,T],L^2(\omega))$ (observe that $\partial_t (I^h m^h) = I^h(dm^h/dt)$):
\[ \sup_h \left( \sup_{t \in [0,T]} \| \nabla (I^h m^h)(t) \|_{L^2(\omega)}^2 + \int_0^T \| \partial_t (I^h m^h) (s) \|_{L^2(\omega)} ds \right) < +\infty. \]
As this is valid for all $T \ge 0$, we can extract via a diagonal argument a weak limit $m \in L^\infty_{\loc}([0,+\infty),H^1(\omega)) \cap \dot H^1_{\loc}([0,+\infty),L^2(\omega))$ (up to a subsequence that we still denote $m^h$) in the following sense
\begin{align}
I^h m^h &\stackrel{*}{\tendf} m \quad *\text{-weakly in } L^\infty_{\loc}([0,+\infty),H^1(\omega)), \\
\partial_t (I^h m^h) & \tendf \partial_t m \quad  \text{weakly in } L^2_{\loc}([0,\infty),L^2(\omega)),
\label{conv:Im2} \\
I^h m^h & \to m \quad \text{a.e.} \label{conv:Im3}
\end{align}
By compact embedding, the following strong convergence also holds:
\begin{align*}
I^h m^h & \to m \quad \text{strongly in } L^2_{\mr{loc}}([0,+\infty), L^2(\omega)).
\end{align*}
Then it follows that for all $t \ge 0$,
\begin{align*}
 \tilde{E}_\delta(m(t)) & \le \liminf_{h \to 0}  \tilde E_\delta (I^h m^h (t)) = \liminf_{h \to 0}  E^h(m^h(t)) \\
 & \le \liminf_{h \to 0}  E^h(m^h(0)) \exp \left(  \frac{4}{\alpha \beta} \int_0^t \| v(s) \|_{L^\infty} ^2 ds \right) \\
& \le { \tilde{E}_\delta(m_0)} \exp \left(  \frac{4}{\alpha \beta} \int_0^t \| v(s) \|_{L^\infty}^2 ds \right).
\end{align*}
 This is the energy dissipation inequality.
 
Observe that if $\varphi$ is a test function, then $\nabla^h \varphi \to \nabla \varphi$ in $L^2$ (strongly). Using \eqref{conv:Im3}, it follows classically (cf \cite[p. 224]{Lad73}) that
\[ m^h \to m \quad \text{strongly in } L^2_{\mr{loc}}([0,+\infty), L^2(\omega)). \]
Therefore $|m|=1$ a.e  and
\begin{align*}
 \nabla^h m^h &\tendf \nabla m \quad \text{weakly in } L^2_{\loc}([0,+\infty, L^2(\omega)).
\end{align*}
From there, arguing in the same way, it follows that
\[ m^h \times \Delta^h m^h = \nabla^h \cdot (m^h \times \nabla^h m^h)  \tendf \nabla \cdot (m \times \nabla m) = (m \times \Delta m) \quad  \text{weakly in } \q D'((0,+\infty) \times \omega). \]
Also notice that $\ds \partial_t (I^h m^h) = I^h \frac{dm^h}{dt}$. Hence
\[ \partial_t m^h \tendf \partial_t m \quad  \text{weakly in } L^2_{\loc} ([0,\infty), L^2(\omega)). \]
We can now deduce the convergences of the other nonlinear terms in the distributional sense: 
\begin{align*} 
m^h \times \partial_t m^h & \tendf m \times \partial_t m \quad \text{weakly in } \q D'((0,+\infty) \times \omega), \\
m^h \times ( v \cdot \nabla^h) m^h & \tendf m \times (v \cdot \nabla) m \quad \text{weakly in } \q D'((0,+\infty) \times \omega),
\end{align*}
and
\begin{align*}
\MoveEqLeft m^h \times (m^h \times ( v \cdot \nabla^h) m^h)  =  - (v^h \cdot \nabla^h) m^h \\
& \qquad \tendf  - (v \cdot \nabla) m = m \times (m \times (v \cdot \nabla) m) \quad \text{weakly in }  \q D'((0,+\infty) \times \omega).
\end{align*}
It remains to consider the nonlocal term. As $\nabla^h m^h \tendf \nabla m$ weakly in $L^2_{\loc} ([0,+\infty) , L^2(\omega))$, we have
\[ |\nabla|^{-1} \nabla^h \cdot m^h{}' \tendf |\nabla|^{-1} \nabla \cdot  m' \quad \text{weakly in } L^2_{\loc}([0,\infty), H^{1} (\omega)). \]
But from \eqref{S_cont}, and noticing that $S^h \varphi \to \varphi$ in $H^1$ strongly for any test function $\varphi$ and as $S^h$ is $L^2$-self adjoint, we infer that
\[ S^h |\nabla|^{-1} \nabla^h \cdot m^h{}' \tendf |\nabla|^{-1} \nabla \cdot  m' \quad \text{weakly in } L^2_{\loc}([0,\infty), H^{1} (\omega)), \]
and similarly,
\[ \q P^h(m^h{}') = - \nabla^h S^h( |\nabla|^{-1} \nabla^h \cdot m^h{}') \tendf - \nabla |\nabla|^{-1} \nabla \cdot m =  \q P ( m') \quad \text{weakly in } L^2_{\loc} ([0,\infty), L^2(\omega)). \]
Recalling that $	m^h \to m$ strongly in $L^2_{\mr{loc}}([0,\infty), L^2(\omega))$, we deduce by weak-strong convergence that 
\[ m^h \times \q H^h(m^h) \tendf m \times \q H (m)  \quad \q D' ([0,\infty) \times \omega). \]
This shows that $m$ satisfies \eqref{eq:LLG} on $[0,\infty) \times\omega$ in the sense of Definition \ref{def:weak}. 
\end{proof}

\appendix

\section{A uniform estimate}

For $\eps>0$ small, we consider the \textbf{full} cell ${\mathcal C}=(0,\eps^\beta)^2\subset \R^2$ with $\nu$ (resp. $\tau$) the unit outer normal vector (resp. the tangent vector) at $\partial {\mathcal C}$ and a boundary data $g_\eps\in H^1(\partial {\mathcal C}, \R^2)$ with {$|g_\eps|\leq 1$} on $\partial {\mathcal C}$.  We recall the definition of the Ginzburg-Landau energy density
\[ \den(u)= |\nabla u|^2+\frac{1}{\eps^2} (1-|u|^2)^2 \quad \textrm{for} \quad u\in H^1({\mathcal C},\R^2).\]
Let $u_\eps \in H^1({\mathcal C},\R^2)$ be a minimizer of the variational problem
\[ \min \left\{ \int_{\mathcal C} \den(u)\, dx\, :\, u=g_\eps\, \, \textrm{ on }\partial {\mathcal C}\right\} .\]
In the spirit of Bethuel, Brezis and H\'elein \cite{BBH_CalcVar}, it will be proved that $|u_\eps|$ is uniformly close to $1$ as $\eps\to 0$ under
certain energetic conditions. The same argument is used in \cite{Ignat-Kurzke}:\footnote{Theorem \ref{thm_main_GL} is an improvement of the results in \cite{BBH_CalcVar} in the case where the energy of the boundary data $g_\eps$ is no longer uniformly bounded.}

\begin{thm} \label{thm_main_GL}
Let $\beta\in (0,1)$. Let $\ka=\ka(\eps)>0$ be such that $\ka=o( |\log \eps|)$ as $\eps\to 0$. Assume that there exists $K_0>0$ such that 
\be
\label{asum1}
\int_{\partial {\mathcal C}} \bigg(|\partial_\tau g_\eps|^2+\frac{1}{\eps^2} (1-|g_\eps|^2)^2\bigg) \, d\Hh^1\leq \frac{K_0\kappa}{\eps^\beta}\quad \textrm{and} \quad \int_{{\mathcal C}} \den(u_\eps)\, dx\leq K_0\kappa,\ee
for all 
$\eps\in (0,\frac 1 2)$.
Then there exist $\eps_0(\beta)>0$ and $C(K_0)>0$ such that for all $0<\eps\leq \eps_0$ we have 
\[ \sup_{{\mathcal C}}\big||u_\eps|-1\big|\leq C \left( \frac{\ka}{|\log \eps|} \right)^{\frac16-},\] where $\frac 16-$ is any fixed positive number less that $\frac 16$. In particular, $|g_\eps|\geq 1/2$ on $\partial {\mathcal C}$ and $\degr(g_\eps, \partial {\mathcal C})=0$.

\end{thm}

\begin{rem}
In the setting of the proof of Theorem \ref{thm_comp} we take $\ka=1/(\delta|\log \delta|)$.

\end{rem}


The proof of Theorem \ref{thm_main_GL} is done by using the following results:

\begin{lem}
\label{lem_normal}
Under the hypothesis of Theorem \ref{thm_main_GL}, we have
$$\int_{\partial {\mathcal C}} \den(u_\eps)\, dx\leq \frac{C\,K_0\kappa}{\eps^\beta},$$
where $C>0$ is some universal constant. Up to a change of $K_0$ in Theorem \ref{thm_main_GL}, we will always assume that the above $C=1$.
\end{lem}
\proof{} Since $u_\eps$ is a minimizer of $\den$, then $u_\eps$ is a solution of
\begin{equation}\label{eq:euler}-\Delta u_\eps=\frac{2}{\eps^2}u_\eps(1-|u_\eps|^2) \quad \textrm{in} \quad  {\mathcal C}.\end{equation} 
We use the Pohozaev identity for $u_\eps$. More precisely, multiplying the equation by $(x-x_0)\cdot \nabla u_\eps$ and integrating by parts,
we deduce:
\begin{align}
\nonumber
\MoveEqLeft \bigg|\frac{1}{\eps^2}\int_{ {\mathcal C}} u_\eps(1-|u_\eps|^2) \cdot \bigg((x-x_0)\cdot \nabla u_\eps\bigg)\, dx \bigg|\\
\label{unu}&=\bigg|\frac{1}{2\eps^2}\int_{ {\mathcal C}} (1-|u_\eps|^2)^2 \, dx -\frac{1}{4\eps^2}\int_{\partial {\mathcal C}} (x-x_0)\cdot \nu (1-|g_\eps|^2)^2 \, d\Hh^1\bigg| \stackrel{\eqref{asum1}}{\leq} CK_0 \kappa,\\
\nonumber
\MoveEqLeft \int_{ {\mathcal C}} \Delta u_\eps \cdot \bigg((x-x_0)\cdot \nabla u_\eps\bigg)\, dx \\
\label{doi}&=\int_{ \partial {\mathcal C}} \bigg( -\frac 1 2 (x-x_0)\cdot \nu |\nabla u_\eps|^2
+\frac{\partial u_\eps}{\partial \nu}\cdot \frac{\partial u_\eps}{\partial (x-x_0)} 
\bigg)\, d\Hh^1,
\end{align}
where $\ds \frac{\partial u_\eps}{\partial (x-x_0)}=\nabla u_\eps\cdot (x-x_0)$. For $x\in \partial {\mathcal C}$, we have $x-x_0=\eps^\beta(\nu+s\tau)$ with $s\in (-1,1)$, $u_\eps(x)=g_\eps(x)$ and we write (as complex numbers)
$\nabla u_\eps=\nabla u_{1,\eps}+i\nabla u_{2,\eps}=\frac{\partial u_\eps}{\partial \nu}\nu+\frac{\partial g_\eps}{\partial \tau}\tau$ on $\partial {\mathcal C}$. By \eqref{eq:euler}, \eqref{unu} and \eqref{doi}, it follows by Young's inequality:
$$\frac{CK_0\kappa}{\eps^\beta} \geq \int_{ \partial {\mathcal C}} \bigg(\frac 1 2 \big|\frac{\partial u_\eps}{\partial \nu}\big|^2
-\frac 1 2 \big|\frac{\partial g_\eps}{\partial \tau}\big|^2+s \frac{\partial u_\eps}{\partial \nu}\cdot \frac{\partial g_\eps}{\partial \tau}\bigg)\, d\Hh^1 
\geq \int_{ \partial {\mathcal C}} \bigg(\frac 1 4 \big|\frac{\partial u_\eps}{\partial \nu}\big|^2
-\frac 3 2 \big|\frac{\partial g_\eps}{\partial \tau}\big|^2\bigg)\, d\Hh^1.$$
Therefore, by \eqref{asum1}, we deduce that $\int_{ \partial {\mathcal C}}  \big|\frac{\partial u_\eps}{\partial \nu}\big|^2\, d\Hh^1\leq \frac{CK_0\kappa}{\eps^\beta}$ and the conclusion follows.
\qed

\bigskip

In the following, we denote by $T(x,r)$ the square centered at $x$ of side of length $2r$.

\begin{lem}
Fix $1>\bun>\bd>\beta>0$. Under the hypothesis of Theorem \ref{thm_main_GL}, there exist $\eps_0=\eps_0(\bd,\beta)>0$ and $C=C(K_0)>0$ such that for every $x_0\in {\mathcal C}$ and all $0<\eps\leq \eps_0$, we can find $r_0=r_0(\eps)\in (\eps^{\bun}, \eps^{\bd})$
such that 
\be
\label{ineq1}
\int_{\partial \big(T(x_0,r_0)\cap {\mathcal C}\big)} \den(u_\eps)\, d\Hh^1\leq \frac{C\kappa}{r_0 |\log \eps|}.\ee
Moreover, we have 
\be \label{ineq2}
\frac{1}{\eps^2} \int_{T(x_0,r_0)\cap {\mathcal C}} (1-|u_\eps|^2)^2\, dx \leq \frac{\tilde{C}\ka}{|\log \eps|}\ee
for some $\tilde{C}>0$ depending on $K_0$.
\end{lem}    

\proof{} We distinguish two steps:

\medskip

{\it Step 1. Proof of \eqref{ineq1}.} Fix $\eps_0\in (0, \frac 12)$ (depending on $\bd-\beta$) such that $\eps_0^{\bd-\beta} |\log \eps_0|\leq 1/2$.
Assume by contradiction that for every $C\geq K_0$ there exist $x\in {\mathcal C}$ and $\eps\in (0, \eps_0)$ such that for every $r\in (\eps^{\bun}, \eps^{\bd})$, we have 
\[ \int_{\partial \big(T(x_0, r)\cap {\mathcal C}\big)} \den(u_\eps)\, d\Hh^1\geq \frac{C\kappa}{r |\log \eps|}. \]
By \ Lemma \ref{lem_normal}, we have 
\[ \int_{\partial {\mathcal C}} \den(u_\eps)\, d\Hh^1\leq \frac{ K_0\kappa}{\eps^\beta}\leq \frac{K_0\kappa}{2 \eps^{\bd} |\log \eps|}\leq \frac{C\kappa}{2r |\log \eps|}, \quad \forall r\in (\eps^{\bun}, \eps^{\bd}) .\]
Therefore, we deduce that
\[ \int_{\partial T(x_0, r)\cap {\mathcal C}} \den(u_\eps)\, d\Hh^1\geq \int_{\partial  \big(T(x_0, r)\cap {\mathcal C}\big)} \den(u_\eps)\, d\Hh^1 - \int_{\partial {\mathcal C}} \den(u_\eps)\, d\Hh^1\geq \frac{C\kappa}{2r |\log \eps|}. \]
Integrating in $r\in (\eps^{\bun}, \eps^{\bd})$, we obtain
\begin{align*}
\MoveEqLeft K_0\ka\stackrel{\eqref{asum1}}{\geq} \int_{{\mathcal C}} \den(u_\eps)\, dx\geq \int_{T(x_0, \eps^{\bd})\cap {\mathcal C}} \den(u_\eps)\, dx \\
& \geq \int_{\eps^{\bun}}^{\eps^{\bd}} 
dr \int_{\partial T(x_0, r)\cap {\mathcal C}} \den(u_\eps)\, d\Hh^1\geq \frac{C(\bun-\bd)\ka}{2}
\end{align*}
which is a contradiction with the fact that $C$ can be arbitrary large.

\bigskip

{\it Step 2. Proof of \eqref{ineq2}.} Let $x_0\in {\mathcal C}$. 
We use the same argument as at Lemma \ref{lem_normal} involving a Pohozaev identity for the solution $u_\eps$ of
\eqref{eq:euler} in the domain $$\cd:=T(x_0, r_0)\cap {\mathcal C}$$ where $r_0$ is given at \eqref{ineq1}. Multiplying the equation by $(x-x_0)\cdot \nabla u_\eps$ and integrating by parts,
we deduce:
\begin{align*}
\MoveEqLeft \int_{ \cd} -\Delta u_\eps \cdot \bigg((x-x_0)\cdot \nabla u_\eps\bigg)\, dx \\
&=\int_{ \partial \cd} \bigg( \frac 1 2 (x-x_0)\cdot \nu |\nabla u_\eps|^2
-\frac{\partial u_\eps}{\partial \nu}\cdot \frac{\partial u_\eps}{\partial (x-x_0)} 
\bigg)\, d\Hh^1,\\
\MoveEqLeft \frac{1}{\eps^2}\int_{ \cd} u_\eps(1-|u_\eps|^2) \cdot \bigg((x-x_0)\cdot \nabla u_\eps\bigg)\, dx \\
&=\frac{1}{2\eps^2}\int_{ \cd} (1-|u_\eps|^2)^2 \, dx -\frac{1}{4\eps^2}\int_{\partial \cd} (x-x_0)\cdot \nu (1-|u_\eps|^2)^2 \, d\Hh^1.
\end{align*}
Since $|x-x_0|\leq \sqrt{2}r_0$ on $\partial \cd$, by \eqref{ineq1},
we deduce that \eqref{ineq2} holds true.
\qed

\begin{lem}\label{lemma:maximum} Under the hypothesis of Theorem \ref{thm_main_GL}, we have that 
$\|u_\eps\|_{L^\infty({\mathcal C})}\leq 1$ and \[
|u_\eps(x)-u_\eps(y)|\le C\left( \frac{|x-y|}{\eps}+ \frac{|x-y|^{\frac12-}}{\eps^{\frac12-}}\right), \quad \forall x, y\in {\mathcal C},
\]
where $C\geq 1$ is a universal constant (independent of $K_0$) and $\frac 12-$ is some positive number less than $\frac 12$.
\end{lem}

\begin{rem}
Unlike \cite{BBH_CalcVar}, the estimate $\|\nabla u_\eps\|_{L^\infty({\mathcal C})}\leq C/\eps$ does not hold in general here since it might already fail for the boundary data $g_\eps$ (due to \eqref{asum1}).
Therefore, the estimate given by Lemma \ref{lemma:maximum} is the natural one in our situation.
\end{rem}

\proof{}
Let $\rho=1-|u_\eps|^2$. Then \eqref{eq:euler} implies that $-\Delta \rho+\frac{4}{\eps^2}|u_\eps|^2\rho\geq 0$ in ${\mathcal C}$ and $\rho=1-|g_\eps|^2\geq 0$ on $\partial {\mathcal C}$.
Thus, the maximal principle implies that $\rho\geq 0$, i.e., $|u_\eps|\leq 1$ on $\partial {\mathcal C}$. 
For the second estimate, we do the rescaling $U(x)=u_\eps(\eps^\beta x)$ for $x\in\Om_0:=(0,1)^2$ and $G(x)=g_\eps(\eps^\beta x)$ for $x\in \partial \Om_0$ and get the equation
$$\textrm{$-\Delta U=\frac{2}{\eps^{2(1-\beta)}}U(1-|U|^2) \quad \textrm{in} \quad  \Om_0$ with $U=G$ on $\partial \Om_0$.}$$ Then we write $U=V+W$ with $-\Delta V=\frac{2}{\eps^{2(1-\beta)}}U(1-|U|^2) \quad \textrm{in} \quad  \Om_0$ and $V=0$ on $\partial \Om_0$ and $\Delta W=0$ in $\Om_0$ with  
$W=G$ on $\partial \Om_0$. In particular, $-\Delta |W|^2=-2|\nabla W|^2\leq 0$ in $\Omega_0$; since $|W|\leq 1$ on $\partial \Om_0$, the maximal principle implies that $|W|\leq 1$ in $\Om_0$. Due to $|U|\leq 1$, we deduce that $|V|\leq 2$ in $\Om_0$. Using the Gagliardo-Nirenberg inequality, we have
\[
\|\nabla V\|_{L^\infty(\Om_0)}\le C_0 \|V\|_{L^\infty(\Om_0)}^{\frac12} \|\Delta V\|_{L^\infty(\Om_0)}^{\frac12},
\]
so that we obtain
\[ \|\nabla V\|_{L^\infty(\Om_0)}\leq C/\eps^{1-\beta}. \]
In order to have the $C^{0,1/2-}$ estimate for $W$, we start by noting that 
\[ \int_{\partial \Om_0} |\partial_\tau G|^2 \, d\h^1=\eps^\beta \int_{\partial {\mathcal C}} |\partial_\tau g_\eps|^2 \, d\h^1\stackrel{\eqref{asum1}}{\leq} K_0 \kappa. \]
So, by regularity theory for harmonic 
functions, we deduce:\footnote{Let us consider for simplicity the following $2D$ situation: $\Delta W=0$ for $x_2\neq 0$ and $W=G$ for $x_2=0$. Passing in Fourier transform in $x_1$, we obtain that ${\mathcal F}(W)(\xi_1, x_2)=e^{-|\xi_1|\, |x_2|} {\mathcal F}(G)(\xi_1)$. Therefore, the Fourier transform in both variables of $\R^2$ of $W$ is given by 
 $\hat{W}(\xi)={\mathcal F}(G)(\xi_1)\int_{\R} e^{-i\xi_2 x_2}  e^{-|\xi_1| |x_2|} \, dx_2={\mathcal F}(G)(\xi_1)\frac{|\xi_1|}{|\xi|^2}$
because ${\mathcal F}(x_1\mapsto \frac 1{1+x_1^2})(\xi_1)=e^{-|x_1|}$. Therefore, $\|W\|_{\dot{H}^{\frac32-}(\R^2)}\sim \|G\|_{\dot{H}^{1-}
(\R)}$.} 
$$ \|W\|_{\dot{H}^{\frac32-}(\Om_0)}\leq C_0\|G\|_{\dot{H}^{1-}(\partial \Om_0)}\leq C (K_0 \kappa)^{1/2}. $$
By Sobolev embedding $H^{\frac32-}(\Om_0)\subset C^{0,\frac12-}(\Om_0)$, it follows:
$$
|W(x)-W(y)|\le C|x-y|^{\frac12-} \|W\|_{\dot{H}^{\frac32-}(\Om_0)}\le C( K_0\kappa )^{1/2} |x-y|^{\frac12-},\quad \forall (x,y) \in \Omega_0^2.
$$
Therefore, we obtain
\begin{equation*} \begin{split}|U(x)-U(y)|&\leq |V(x)-V(y)|+ |W(x)-W(y)|\\&\leq C\left(\frac{ |x-y|}{\eps^{1-\beta}} + \frac{|x-y|^{\frac12-}}{\eps^{\frac{1-\beta}{2}-}} (K_0 \kappa)^{1/2} \eps^{\frac{1-\beta}{2}-} \right),
\quad \forall (x,y) \in \Omega_0^2. \end{split}\end{equation*}  Scaling back, we obtain the desired estimate for $u_\eps$ in ${\mathcal C}$ since $(K_0 \kappa)^{1/2} \eps^{\frac{1-\beta}{2}-}=o(1)$. 
\qed

\proof[Proof of Theorem \ref{thm_main_GL}.] We will show that
\[ \||u_\eps|^2-1\|_{L^\infty({\mathcal C})}\leq C \left( \frac{\ka}{|\log \eps|} \right)^{\frac16-}. \]
Let $x_0\in {\mathcal C}$ such that $|u_\eps(x_0)|<1$. Set $0<A<1$ such that
\[ 2C(2A+(2A)^{\frac12-})=\frac{(1-|u_\eps(x_0)|^2)}{2}>0, \]
where $C$ is given by Lemma \ref{lemma:maximum}. In particular, $A^{\frac12-}\geq A\geq {C}_1 (1-|u_\eps(x_0)|^2)^{2+} $. By 
Lemma~\ref{lemma:maximum}, we obtain for any $y\in T(x_0,A\eps)\cap {\mathcal C}$: $|y-x_0|\leq 2A\eps$ and
\[ 1-|u_\eps(y)|^2\geq 1-|u_\eps(x_0)|^2-2C(2A+(2A)^{\frac12-})=\frac{1-|u_\eps(x_0)|^2}{2}. \]
Hence, for small $\eps$, we have $A\eps<\eps\leq \eps^{\beta_1}\leq r_0$ (with $r_0$ given in \eqref{ineq1}) and
\begin{align*}
\MoveEqLeft \frac{\tilde{C}\kappa \eps^2}{|\log \eps|}\stackrel{\eqref{ineq2}}{\geq} \int_{T(x_0, A\eps)\cap {\mathcal C}}(1-|u_\eps(y)|^2)^2\,dy \nonumber \\
& \geq \frac 1{16} A^2\eps^2 (1-|u_\eps(x_0)|^2)^2 = \frac 1{16}C_1^2 \eps^2 (1-|u_\eps(x_0)|^2)^{6+}. \nonumber
\end{align*}
Thus, we conclude that
\[ (1-|u_\eps(x_0)|^2)^{6+}\leq  \hat{C}\frac{\ka}{|\log \eps|}. \qedhere\]


\medskip

\medskip

\textbf{Acknowledgments.} The last author is partially supported by the French ANR projects SchEq ANR-12-JS-0005-01 and GEODISP ANR-12-BS01-0015-01.


\begin{thebibliography}{99}
\bibitem{AlSo} F. Alouges, A. Soyeur, \emph{On global weak solutions for Landau-Lifshitz equations}, Nonlinear Analysis, Theory, Methods \& Applications {\bf 18} (1992), 1071--1084.
\bibitem{BBH_CalcVar} F. Bethuel, H. Brezis, F. H\'elein, {\it Asymptotics for the minimization of a Ginzburg-Landau functional}, Calc. Var. Partial Differential Equations {\bf 1} (1993), 123-148.
\bibitem{BBH-book} F. Bethuel, H. Brezis, F. H\'elein, {\it Ginzburg-Landau vortices}. Progress in Nonlinear Differential Equations and their Applications, 13. Birkh\"auser Boston, 1994.
\bibitem{CMA} A. Capella, C. Melcher, F. Otto, \emph{Effective dynamics in feromagnetic thin films and the motion of N\'eel walls}, Nonlinearity \textbf{20} (2007), 2519--2537.
\bibitem{DKMO_reduced} A. DeSimone, R. V. Kohn, S. M\"uller, F. Otto, \emph{A reduced theory for thin-film micromagnetics}, 
Comm. Pure Appl. Math. {\bf 55} (2002),  1408-1460.
\bibitem{DeSKMO} A. DeSimone, R. V. Kohn, S. M\"uller, F. Otto, \emph{Recent analytical developments in micromagnetics}, in: Giorgio Bertotti, Isaak Mayergoyz (Eds.), The science of Hysteresis, vol. 2, Elsevier, Academic Press, 2005, 269--381 (Chap. 4). 
\bibitem{Gilbert} T. L. Gilbert, \emph{A lagrangian formulation of gyrokinetic equation of the magnetization field}, Phys. Rev. \textbf{100} (1955), 1243.
\bibitem{I} R. Ignat, \emph{A $\Gamma$-convergence result for N\'eel walls in micromagnetics}, Calc. Var. Partial Differential Equations \textbf{36} (2) (2009), 285--316.
\bibitem{IGNAT-XEDP} R. Ignat, {\it A survey of some new results in ferromagnetic thin films}, S\'emin. \'Equ. D\'eriv. Partielles, Ecole Polytech., Palaiseau, 2009. 
\bibitem{IK} R. Ignat, H. Kn\"upfer, \emph{Vortex energy and 360$^\circ$-N\'eel walls in thin-film micromagnetics}, Comm. Pure Appl. Math. \textbf{63} (2010), 1777--1724.
\bibitem{Ignat-Kurzke} R. Ignat, M. Kurzke, {\it An effective model for boundary vortices in thin-film micromagnetics}, in preparation.
\bibitem{IO1} R. Ignat, F. Otto, \emph{A compactness result in thin-film micromagnetics and the optimality of the N\'eel wall}, J. Eur. Math. Soc. (JEMS)  \textbf{10} (4) (2008), 909--956.
\bibitem{IO2} R. Ignat, F. Otto, \emph{A compactness result for Landau state in thin-film micromagnetics}, Ann. I. H. Poincar\'e \textbf{28} (2011), 247--282.
\bibitem{Kohn-Slastikov} R.V. Kohn, V. Slastikov, {\it Another thin-film limit of micromagnetics}, Arch. Ration. Mech. Anal. {\bf 178} (2005), 227-245.
\bibitem{KMM} M. Kurzke, C. Melcher, R. Moser, \emph{Vortex motion for the Landau-Lifshitz-Gilbert equation with spin-transfer torque}, SIAM J. Math. Anal. \textbf{43} (3) (2011), 1099--1121.
\bibitem{Lad73} O. Ladysenskaya, \emph{The boundary value problems of mathematical Physics}. Nauka, Moscow (1973) (in Russian); English translation: \emph{Applied mathematical sciences}, vol. 49. Springer, Berlin (1985).
\bibitem{Landau-Lifshitz} L. D. Landau, E. Lifshitz, \emph{On the theory of the dispersion of magnetic permeability in ferromagnetic bodies}, Phys. Z. Sovietunion \textbf{8} (1935), 153--169.
\bibitem{Melcher} C. Melcher, \emph{Thin-film limits for Landau-Lifschitz-Gilbert equations}, SIAM J. Math. Anal. \textbf{42}, no. 1, 519--537.
\bibitem{MS98} S. Müller, M. Struwe, \emph{Spatially discrete wave maps on (1 + 2)-dimensional space-time}, Topol. Methods Nonlinear Anal. \textbf{11} (1998), 295--320.
\bibitem{Simon} J. Simon, \emph{Compact sets in the space $L^p(0,T;B)$},
Ann. Mat. Pura Appl. \textbf{146} (4) (1987), 65–96.
\bibitem{Th} A. Thiaville, Y. Nakatani, J. Miltat, Y. Suzuki, \emph{Micromagnetic understanding of
current-driven domain wall motion in patterned nanowires}, Europhys. Lett. \textbf{69} (2005),
article 990.
\bibitem{ZL} S. Zhang, Z. Li, \emph{Roles of nonequilibrium conduction electrons on the magnetization dynamics of ferromagnets}, Phys. Rev. Lett. \textbf{93} (2004), article 127204.
\end{thebibliography}
\end{document}